\documentclass[11pt]{amsart}

\usepackage{amsgen,amsthm,amstext,amsbsy,amsopn,amsfonts,amssymb, enumerate,xcolor}

\usepackage{amsmath}

\newcommand\la{\langle}
\newcommand\ra{\rangle}

\newcommand\hh{{\mathfrak h}}

\newcommand\nn{{\mathfrak n}}

\newcommand\vv{{\mathfrak v}}

\newcommand\zz{{\mathfrak z}}

\newcommand\CC{\mathbb C}

\newcommand\NN{\mathbb N}

\newcommand\RR{\mathbb R}
\newcommand\ZZ{\mathbb Z}

\newcommand\Auto{\operatorname{Auto}}

\newcommand\Iso{\operatorname{Iso}}

\DeclareMathOperator{\En}{En}

\DeclareMathOperator{\sgn}{sgn}

\theoremstyle{plain}

\newtheorem{thm}{Theorem}[section]
\newtheorem{lem}[thm]{Lemma}
\newtheorem{prop}[thm]{Proposition}
\newtheorem{cor}[thm]{Corollary}

\theoremstyle{definition}
\newtheorem{defn}[thm]{Definition}
\newtheorem{rem}[thm]{Remark}
\newtheorem{example}[thm]{Example}

\newcounter{casenum}

\begin{document}
	
	\title[Magnetic geodesics on Heisenberg nilmanifolds]
	{Closed Magnetic geodesics on Heisenberg nilmanifolds}

	\author{Gabriela P. Ovando, Mauro Subils}
	
	\thanks{{\it (2000) Mathematics Subject Classification}: 53C99, 70G65, 70F17, 22E25}
	
	\thanks{{\it Key words and phrases}: Closed magnetic trajectories,   Heisenberg Lie group, Heisenberg nilmanifolds, Ma\~n\'e critical value.
	}
	
	\thanks{Partially supported by  ANPCyT, SCyT (UNR)}
	
	\address{ Departamento de Matem\'atica, ECEN - FCEIA, Universidad Nacional de Rosario.   Pellegrini 250, 2000 Rosario, Santa Fe, Argentina.}
	
	\
	
	\email{gabriela@fceia.unr.edu.ar}
	
	\email{subils@fceia.unr.edu.ar}
	

	\begin{abstract}

	In this work we study the existence of closed magnetic geodesics on three-dimensional Heisenberg nilmanifolds for every left-invariant Lorentz force. Our first objective is to establish the existence of closed contractible magnetic geodesics on $H_3$. Once the invariant magnetic field is induced to a  compact quotient $M=\Lambda \backslash H_3$, we study magnetic geodesics on $M$. Firstly, we determine conditions on a lattice $\Lambda  \subset  H_3$ to ensure that a given magnetic geodesic projects to a closed curve on $M$.
	In particular, we prove that for {\em any} energy level below the Mañé critical value there always exists a contractible closed magnetic geodesic on the compact manifold $M$. On the other hand, we show that closed magnetic geodesics do not necessarily exist in every homotopy class.
	Finally, we present examples of compact quotients $\Gamma_k\backslash H_3$ that admit infinitely many closed magnetic trajectories, as well as examples for which no closed non-contractible magnetic trajectories exist for a given left-invariant Lorentz force.
\end{abstract}

\maketitle

\setcounter{tocdepth}{1}
\tableofcontents

\noindent\section{Introduction}

In the present paper we work on the Heisenberg Lie group of dimension three $H_3$ (and compact quotients $\Lambda \backslash H_3$)  with solutions of the following Equation
\begin{equation}\label{mageq}
	\nabla_{\gamma'}{\gamma'}= q F\gamma'
\end{equation}
where $\nabla$ is the corresponding Levi-Civita connection and $F$ is a skew-symmetric $(1,1)$-tensor such that the corresponding 2-form $\omega_F:=g(F\cdot ,\cdot)$ is closed. This 2-form is called a {\em magnetic field}, while $F$ is referred as a {\em Lorentz force}. Solutions of Equation \eqref{mageq}, are called {\em magnetic trajectories or magnetic geodesics}.

This mathematical formalism plays an important role in several situations such as the study of the motion of a charged particle in the presence
of a time-independent magnetic field in three-dimensional space, in the classical magnetostatic theory; see for instance \cite{Thi}. A model for this system  consists of a Riemannian manifold where the trajectory of a particle subject to a force is described by an equation of the form
above, according to Maxwell's Equations.   

One can find different approaches in the literature to the question of existence or non-existence of closed magnetic geodesics: the theory of dynamical systems using methods from symplectic geometry in \cite{Ar,Gi,Sl}, the Morse-Novikov theory \cite{NT, Tai} and Aubry-Mather's theory \cite{CMP}. Another approach was used in \cite{Sc} by studying  the zeros of a certain vector field. Various results were stated depending on the energy level or the nature of the magnetic field: if it is exact or not (called a {\em monopole}). For exact magnetic flows in any dimension, Hofer and Viterbo proved the existence of periodic orbits for high energy levels \cite{HV}. For low energy levels, the Mañ\'e critical value plays an important role related to the existence question and reveals different behaviour of the flow \cite{Ma,CMP}.

To the best of the authors’ knowledge, the existence of contractible closed magnetic geodesics at every energy level below the Mañ\'e critical value is not guaranteed. Thus, in light of the results above, the search for periodic orbits was divided into three realms of high-, low-, and intermediate-energy levels. In the intermediate case, there is generally no information about the existence of such orbits. On the other hand, these situations must be contrasted  with the magnetic field, if it is exact or not.  For magnetic monopoles one cannot expect to find periodic orbits in all energy levels.

Our main contribution to the question of existence of closed magnetic geodesics on compact spaces is the following result (for a discrete cocompact lattice $\Gamma \subset H_3$), which applies to any invariant magnetic field and whose proof makes use of tools from Lie theory:

\smallskip

{\it {\bf Theorem 1.} Let $\Gamma\backslash H_3$ denote a closed connected Riemannian Heisenberg nilmanifold, and let $\omega$  be a closed 2-form whose pullback to the universal cover $H_3$ gives rise to a left-invariant Lorentz force $F$. Let $c>0$ denote the corresponding  Mañé critical value.
	
	\begin{enumerate}[(i)]
		\item There exists a contractible closed magnetic geodesic with energy $k$ if and only if $0<k\leq c$.
		
		\item There exists a closed magnetic geodesic with energy $k$ and nontrivial homotopy class $0 \neq \lambda\in\Gamma$ if and only if either $k> c$ or, for $k\leq c$, $\lambda$ is $F$-admissible (see Definition \ref{defFadm}).
		
\end{enumerate}	}


This result should be compared with Theorem 1.1 in the works \cite{Me,Os} obtained by Merry and Osuna respectively,  which applies in a more general setting, although in the present particular case we obtain a stronger conclusion.

We consider the standard left-invariant metric on the Heisenberg Lie group $H_3$, which descends to compact quotients $\Gamma \backslash H_3$.  This corresponds to the so-called nilgeometry, one of Thurston’s eight geometries.  For the magnetic equation \eqref{mageq}, we restrict attention to left-invariant magnetic fields (equivalently, Lorentz forces). We note that, locally, there is a correspondence between solutions of the magnetic equation on $H_3$ and on compact nilmanifolds $\Gamma\backslash H_3$. 

In Section \ref{symmetries}, we recall the main results on
 magnetic geodesics on $H_3$ done in \cite{OS3}, since they are necessary for the study of the closeness condition.  There, it was introduced an equivalence relation on the set of Lorentz forces under which, any  invariant Lorentz force $F$ is equivalent to exactly one of the following matrices (in the basis of left-invariant vectors of $\hh_3$):	 
$$F=0,\qquad \quad F_{e_1,\rho}=\left( \begin{matrix}
	0 & -\rho  & -1\\
	\rho & 0 &0 \\
	1 & 0 & 0
\end{matrix}
\right), \quad \mbox{ for } \rho\geq 0\quad \mbox{ or }\quad F_{0,1}=\left( \begin{matrix}
	0 & -1& 0\\
	1 & 0 &0 \\
	0 & 0 & 0
\end{matrix}
\right). 
$$ In canonical coordinates of $\RR^3$, the non-trivial matrices correspond to the following closed 2-forms
$$F_{e_1,\rho} \longleftrightarrow \left(\rho - \frac12 x\right) dx \wedge dy +dx \wedge dz, \qquad \quad F_{0,1}\longleftrightarrow  dx \wedge dy.$$
For the Lorentz force $F_{0,1}$ corresponding to left-invariant exact forms, the magnetic trajectories were obtained by Epstein, Gornet and Mast \cite{EGM} and Munteanu and Nistor in \cite{MN}, while the situation $F_{e_1,0}$ was solved in \cite{OS}. This last case corresponds to invariant harmonic 2-forms. As proved in \cite{OS3} these magnetic geodesics are solutions of a variational problem, that is, they are extremals of a functional given by
$$A(u) = \int L(t,x,y,z,\dot{x}, \dot{y}, \dot{z}) dt,$$
with $u:[0,T]\to \RR^3$, $u(t)=(x(t), y(t), z(t))$. In fact, in \cite{OS3}  a Lagrangian  was given as
$$L= T  - \theta_{F_{e_1,\rho}}, \qquad { \begin{array}{rcl} 
		\mbox{ where }	T & = & \frac12
		\left(
		\dot{x}^2 + \dot{y}^2 + (\dot{z} +\frac12 (\dot{x}y - x\dot{y}))^2\right),\\ 
		\mbox{ for }	\theta_{F_{e_1,\rho}} & = &  (\frac{\rho}2 y - \frac12 xy)\dot{x} -\frac{\rho}2 x \dot{y} - x \dot{z}, \quad  \theta_{F_{0,1}} = \dot{z} + \frac12 (y\dot{x} - x\dot{y}).
\end{array}}$$

At the end of this section, we give a brief description of the magnetic geodesics corresponding to $F_{e_1,\rho}$, as computed in \cite{OS3}. These curves written as exponential of a curve on the Lie algebra with coordinates  $x(t), y(t), z(t)$ are completely determined by the component $x(t)$ and are classified according to the sign of the discriminant $\Delta$ of a quartic polynomial depending on the initial conditions $x_0,y_0,z_0$. Moreover, the  components can be expressed as rational functions in elliptic functions, whose parameters depend on the roots of this polynomial.

In Section \ref{section5}, we study the existence of closed magnetic trajectories on $H_3$. While for the Lorentz force $F_{0,1}$ periodic trajectories exist only at low energy levels, we prove in Theorem \ref{thenergyperiodic} that, for $F_{e_1,\rho}$, there exist infinitely many periodic magnetic trajectories of any given energy $E>0$ passing through the identity with $x_0\geq 0$. First, the problem is reduced to find trajectories for which $x(t)$ is periodic with period $\omega$ and $y(\omega)=0$. When $\Delta\geq 0$, one shows that $y(\omega)\neq 0$, hence no periodic trajectories arise in this case. When $\Delta<0$, we prove the existence of periodic trajectories for all energy levels. More precisely, we parametrize such trajectories by $(c,d,e)\in(0,\infty)\times(0,1)\times[-1,1]$, express $y(\omega)$ in these variables, and show that for each $(c,e)$ there exists a unique value $d_c$ such that $y(\omega)=0$. The corresponding energy depends bijectively on $c$ and ranges over all positive values. Finally, we introduce an equivalence relation (based on isometries and translation of the parameter) on periodic trajectories and prove  that trajectories with the same energy are equivalent, see  Theorem \ref{Equienergy}.

In Section \ref{section6}, we study closed magnetic trajectories on $\Gamma\backslash H_3$, which correspond to projections of $\lambda$-periodic trajectories (Definition \ref{periodic}) on \(H_3\), where $\lambda\in\Lambda$. We begin by describing, in Lemma \ref{lambdaper}, criteria on the element $\lambda\in \Lambda$ that ensure the existence of $\lambda$-periodic magnetic trajectories for any 2-step nilpotent Lie group satisfying $\dim \zz - \dim \nn'\leq 1$, where $\zz$ denotes the center and $\nn'$ the commutator. Then, Proposition \ref{proplambdaperiod1} characterizes when a magnetic trajectory on $H_3$ is $\lambda$-periodic for some arbitrary $\lambda$ and, on the other hand, Proposition \ref{proplambdaperiod2} establishes the conditions on $\lambda$ for the existence of $\lambda$-periodic magnetic trajectory. In particular, we show that if a lattice $\Lambda$ admits suitable elements $\lambda$, then the nilmanifold $M=\Lambda\backslash H_3$ admits periodic magnetic trajectories for every energy level $E>0$.

In the final section, we compute the Mañé critical value for left-invariant magnetic fields on $\Lambda \backslash H_3$, which are weakly exact, see \cite{Me}. We show that when the magnetic field does not admit an invariant primitive on $H_3$, the critical value is infinite. In the exact case ($F=F_{0,1}$), this value was computed in \cite{EGM}; here we provide an alternative proof. By summarising  results of Sections 3, 4 and 5, we obtain the proof of Theorem 1.


These results on closed magnetic geodesics can  be compared with those of closed  geodesics, which in the context of 2-step nilmanifolds  have been studied by several authors such as DeCoste, DeMeyer, Eberlein, Lee, Mast and Park,  see for instance \cite{dC,dM,Eb,LP,Ma}.

As an application, we exhibit examples of manifolds $\Gamma\backslash H_3$ that admit infinitely many closed non-contractible magnetic trajectories, as well as examples that admit none.

In the Appendix we provide properties of elliptic integrals which are used to determine the existence of closed curves. 





\section{The magnetic equations on the Heisenberg Lie group}	\label{symmetries}
In this section we show features of the magnetic equations. We recall tools given  in \cite{OS3}, where the left-action of a group, that includes the isometries, enables to find symmetries of the equations and therefore to simplify them for solving.

\smallskip

Let $(M,g)$  denote a Riemannian manifold with Levi-Civita connection $\nabla$. Let $\Omega$ be a differentiable 2-form on $M$, then there exists a unique   skew-symmetric $(1,1)$-tensor $F:TM \to TM$  defined by the relation:
$$\Omega_F(U,V)=g(FU,V), \quad \mbox{ for all }\quad  U,V\in \chi(M).$$
Conversely, the relation above defines a 2-form whenever $F$ is given. If the 2-form is closed, then one gets an extra condition on $F$. In such case the tensor $F$ is called a  {\em  Lorentz force}. 


Given a Lorentz force $F$ and a magnetic trajectory $\gamma$, one has
$$\frac{d}{dt}g(\gamma'(t), \gamma'(t))=2g(\nabla_{\gamma'(t)}\gamma'(t), \gamma'(t))=2g(F\gamma'(t), \gamma'(t))=0, $$
implying that magnetic  curves have  constant velocity.  

The {\em energy} of the magnetic trajectory $\gamma$ is defined by the scalar 
\begin{equation}\label{energy}
	\En(\gamma)=\frac{1}{2}g(\gamma'(0), \gamma'(0)).
\end{equation}

Note that a reparametrization of a magnetic curve could not be a solution of Equation \eqref{mageq}. In fact, take the curve $\tau(t)=\gamma(r t)$ with $r\neq 0,\,1$. Then it holds $\tau'(t)=r \gamma'(r t)$,  so that one has 
$$\nabla_{\tau'(t)}\tau'(t)= r^2 F\gamma'(r t),\mbox{ while on the other side } F \tau'(t)=r F \gamma'(r t).$$	

Denote by  $\mathcal{F}$ the set of Lorentz forces on $M$ and by $\mathcal{C}$ the set of all differentiable curves on $M$. 
Let  $\Iso(M)$ be the isometry Lie group of $M$ and let  $\mathbb{R}^*=\RR-\{0\}$ be the multiplicative group. 
There is a left-action of the product group  $\Iso(M) \times \RR^*$ on $\mathcal{F}$ given by:
\begin{equation}\label{action1}
	(\psi, r)\cdot F =r \psi_*\circ  F \circ  \psi^{-1}_*, \quad \mbox{for } \psi \in \Iso(M), r\in \RR^*, F\in \mathcal F, 
\end{equation}
where $\psi_*$ denotes the differential of $\psi$. 
The same group acts on the set $\mathcal{C}$ in the following way:
\begin{equation*}
	((\psi, r)\cdot \gamma)(t)=(\psi\circ\gamma)(rt), \quad \mbox{for } \psi \in \Iso(M),  r\in \RR^*, \gamma \in \mathcal C. 
\end{equation*} 	

Note that  the associated $2$-form of the $(1,1)$-tensor $(\psi, r)\cdot F$ is a multiple of the pullback by $\psi^{-1}$ of $\Omega_F$. This shows that $(\psi, r)\cdot F$ is a Lorentz force:  $(\psi, r)\cdot F\in\mathcal{F}$. It is easy to check that these left-actions are well defined. 

Fix a left-action of a group $H$ on a set $S$. Thus, the {\em isotropy subgroup} for $s\in S$ is the subgroup of $H$ given by $$H_s=\{h\in H \,:\, h\cdot s=s\}, $$
while the {\em orbit} of $s\in S$ is the subset  defined by $H\cdot s=\{h\cdot s \, : \, h\in H\}$. 

\begin{lem} \cite{OS3} \label{lem2} Let $(M,g)$ denote a Riemannian manifold. Let $F$ denote a Lorentz force  on $M$ and consider
	the group $H=\Iso(M)\times\mathbb{R}^{*}$. 
	\begin{enumerate}[(i)]
		\item If $\gamma$ is a magnetic trajectory for the Lorentz force $F$ then for any $(\psi,r)\in\Iso(M)\times\mathbb{R}^{*}$, the curve $(\psi,r)\cdot\gamma$ is a magnetic trajectory for the Lorentz force $(\psi,r)\cdot F$. 
		\item  Let  $H_F$ be the isotropy subgroup for the Lorentz force $F$. Then $H_F$ is a group of symmetries for  the magnetic equation $\eqref{mageq}$. 
		
		Furthermore,  $H_F$ consists of elements $(\psi,r)$, with $r=\pm 1$ and  such that
		\begin{equation}\label{symmetry}\psi_* \circ F \circ \psi_*^{-1}=F\quad \mbox{ or } \quad \psi_* \circ F \circ \psi_*^{-1}= -F.
		\end{equation}
	\end{enumerate}
\end{lem}
The proof of (i) in the  Lemma follows from usual computations, while the second statement is  consequence from the definitions.

To complete the proof, let $F$ be a Lorentz force of $M$ and take $(\psi,r)\in H_F$. So, the linear map $F$ is  skew-symmetric at every $T_pM$. As usual denote by  $\Vert v\Vert$  the norm of a vector $v\in T_pM$:  $\Vert v\Vert=g_p(v,v)^{1/2}$.  It follows that  if  $(\psi, r)$  is an element of  the isotropy subgroup for $F\neq 0$,  one has
\begin{equation*}
	0< \sup_{\Vert v\Vert=1}\Vert F(v)\Vert = \sup_{\Vert v\Vert=1}|r| \, \Vert \psi_* F \psi^{-1}_*(v)\Vert= |r|\sup_{\Vert v\Vert=1} \Vert F \psi^{-1}_*(v)\Vert = |r| \sup_{\Vert v\Vert=1}\Vert F (v)\Vert,  
\end{equation*}
which implies  $r=1$ or $r=-1$. This says that $(\psi,r)\in H_F$ for a Lorentz force $F$ if and only if Equations  \eqref{symmetry} hold.

\begin{rem}\label{presenergy}
	Since $|r|=1$, the action of $H_F$ on the curves preserves the energy level.
\end{rem}







Assume now that the Riemannian manifold is a nilpotent Lie group equipped with a metric $(N,\la\,,\,\ra)$ for which translations on the left by elements of the group are isometries, called a left-invariant metric. Note that a left-invariant metric on a Lie group is determined by its values at the Lie algebra, usually identified with the tangent space at the identity element $T_eN$. 

\begin{example}\label{heis} The lowest-dimensional (non-abelian) nilpotent Lie group is the so-called {\em Heisenberg Lie group} of dimension three, which can be modelled on $\RR^3$.
In usual coordinates $(x,y,z)$, the group law on $H_3$ reads
$$(x_1,y_1, z_1)(x_2,y_2, z_2)=(x_1+x_2, y_1+y_2, z_1+z_2+\frac12 (x_1 y_2- x_2 y_1)), $$
and a basis of left-invariant vector fields is given by the globally defined vector fields on $\RR^3$, evaluated at $p=(x,y,z)$:
$$e_1(p)=\partial_x - \frac{1}2 y \partial_z, \quad  e_2(p)=\partial_y + \frac{1}2 x \partial_z, \quad  e_3(p)=\partial_z,$$
where we make use of $\partial_u$ to denote the partial derivative $\frac{\partial}{\partial u}$.	
The canonical metric is the only left-invariant metric on $H_3$ for which the basis $\{e_1,e_2,e_3\}$ is orthonormal.
	\end{example}

In \cite{Wo} Wolf proved that the group of isometries of such Riemannian manifolds, is given by $\Iso(N,\la\,,\,\ra)=N.\Auto(N)$, where $N$ is identified with  the subgroup of translations on the left by elements of $N$ and $\Auto(N)$ denotes the group of orthogonal automorphisms. In analogy to  the Euclidean space, 
any isometry can be written as
$$L_p\circ \psi, \quad \mbox{ for } \psi\in \Auto(N), p\in N,$$
where $L_p$ denotes a translation on the left by $p\in N$. 

 Moreover $N$ is a normal subgroup of $\Iso(N)$; it holds
$$L_p \circ \psi=\psi \circ L_{\psi(p)}\quad \mbox{ for any } \psi\in \Auto(N), p\in N.$$	 
Consider a left-invariant Lorentz force, that means that it is invariant by translations on the left:
$${L_p}_* \circ F= F\circ {L_p}_* , \quad \mbox{ for any } p\in N. $$
In terms of the action of Lemma \ref{lem2} one gets,   ${L_p}_* \cdot F=F$, saying that the subgroup of left-translations $N$ is contained in the isotropy subgroup $H_F$.

Let $\mathcal F_l$ denote the set of left-invariant Lorentz forces. Thus  the action of the group $\Iso(N)\times \RR^*$ on the set $\mathcal F_l$ is determined by the action of orthogonal automorphisms:
$$(L_p\circ \psi) \cdot F=\psi\cdot  F.$$


In this work, we focus on the Heisenberg Lie group $H_3$ endowed with its canonical metric as in Example \ref{heis}. So, $H_3$ is the simply connected Lie group whose Lie algebra is the Heisenberg Lie algebra $\hh_3$. This is a $2$-step nilpotent Lie algebra spanned by vectors $e_1, e_2, e_3$ satisfying the non-trivial Lie bracket relation $$[e_1,e_2]=e_3.$$ 
 The center of $\hh_3$, which coincides with the commutator,  is spanned by the vector $e_3$ and its orthogonal complement $\vv$ is spanned by $e_1, e_2$. Thus one has the orthogonal decomposition as vector spaces:
	$$\hh_3=\vv \oplus \zz, \quad \mbox{ with }\quad \zz=\RR e_3,\quad  \vv=\zz^\perp.$$

	
The isometry group of the Heisenberg Lie group $(H_3, \la\,,\,\ra)$ is  $\Iso(H_3)= H_3\rtimes\mathrm O(2)$, where $\mathrm O(2)$ is the group of orthogonal automorphisms. The action on $\hh_3$ is explicitly given by
	$$B\cdot (V,Z) = (B(V),\,det(B)Z), \ \  \mbox{ for }B\in \mathrm O(2), V\in \vv, Z\in \zz.$$
	Note that we identify $\RR^2$ with the subspace $\vv=span\{e_1,e_2\}$. And since $H_3$ is simply connected, one makes no distinction between automorphisms of $N$ and $\nn$.

A left-invariant Lorentz force on $H_3$ corresponds to  a skew-symmetric map on $\hh_3$. It is easy to show that such map  $F:\hh_3 \to \hh_3$ has a matricial presentation in the basis $e_1, e_2, e_3$ as 	
\begin{equation} \label{matrixF}
	F=\left( \begin{matrix}
	0 & -\rho  & - \beta\\
	\rho & 0 & -\alpha \\
	\beta & \alpha & 0
\end{matrix}
\right), \quad \alpha, \beta \in \RR. 
\end{equation}
We denote $F$ by $F_{U,\rho}$ with  $U=\beta e_1+ \alpha e_2\in\hh_3$. Also, $F$ corresponds to a unique left-invariant 2-form on $H_3$, that can be describe at $\hh_3$ by $\omega_F=\langle F\cdot, \cdot\rangle$. 

\smallskip

We recall now some definitions about differential forms on a Riemannian manifold. Let $\Omega(M)$ be the space of all differential forms on $M$, $d:\Omega(M)\to \Omega(M)$ denote the exterior derivative, $\star:\Omega(M)\to \Omega(M)$ denote the Hodge star operator and $d^*:\Omega(M)\to \Omega(M)$ the codifferential, see for instance \cite{Wa}. A 2-form $\omega$ is called
\begin{itemize}
	\item  {\em closed} if $d\omega=0$,
	\item  {\em exact} whenever it exists a 1-form such that $d\theta=\omega$,
	\item {\em co-closed} if $d^* \omega=0$, and
	\item {\em harmonic} if it is closed and co-closed. 
\end{itemize}  

If $N$ is a Lie group with Lie algebra $\nn$ endowed with a left-invariant metric and we restrict to left-invariant differential forms, then a Hodge decomposition holds within this class. Let $\nn^*$ denote the dual of $\nn$. Indeed, the operators $d$ and $\star$ act at the Lie algebra level and $d$ has an adjoint operator with respect to the induced metric, namely the codifferential $d^*$. Thus, in the case of $2$-forms, we have: 
$$\Lambda^2\nn^*=d(\Lambda^1\nn^*)\oplus H^2(\nn)\oplus d^*(\Lambda^3\nn^*)$$ 
where $d^*=-\star d \star $ is the codifferential on $3$-forms, $\Lambda^p\nn^*$ corresponds to the space of left-invariant $p$-forms on $N$ and $H^2(\nn)$ the subspace of harmonic $2$-forms.

If $M=H_3$ with the metric presented above, we have that 
$$d(\Lambda^1\hh_3^*)=span\{e^1\wedge e^2\}, \ \ H^2(\hh_3)=span\{e^1\wedge e^3, e^1\wedge e^3\}, \ \ d^*(\Lambda^3\hh_3^*)=\{0\}.$$
So the Lorentz forces in Equation \eqref{matrixF} corresponding to harmonic invariant forms are the ones with $\rho=0$.

On the other hand, since $H_3$ is diffeomorphic to $\RR^3$ the proof of the next assertion is straightforward.

\smallskip

\begin{cor} Any left-invariant 2-form on $H_3$ is exact.
	\end{cor}
 
Indeed, the corresponding 1-form does not need to be invariant. Nonetheless, the Lorentz forces in Equation \eqref{matrixF} with $\alpha=\beta=0$ have a left-invariant primitive, i.e. the left-invariant exact forms.   
	

	

\

Now consider the action of $\Iso(H_3)\times \RR^*$ on the set of left-invariant Lorentz forces $\mathcal F_l$ described in Lemma \ref{lem2}, which can be equivalently viewed as an action on 2-forms. From the definitions, this action on a Lorentz force $F$ is determined by $B\in \mathrm O(2)$ and  $r\in \RR^*$, and is given by
$$(B,r)\cdot F_{U, \rho} = r\det(B)F_{BU,\rho}.$$

The next result describes the orbits on $\mathcal F_l$. 
 Recall that $\mathrm{O}(2)$ has two connected components: $\mathrm{SO}(2)$ which consists of matrices with determinant equals one and $\mathrm{SO}^-(2)$ consisting of matrices with determinant minus one.

\begin{prop} \cite{OS3}\label{orbits} The orbits under the action of $G=\Iso(N)\times \RR^*$ on the set of non-trivial left-invariant Lorentz forces $\mathcal F_l$ are mostly parametrized by $\rho\geq 0$; that is, every point in the following set gives different orbits:
	$$\{F_{e_1, \rho}\}_{\{ \rho \geq 0\}}\cup \{F_{0,1}\}.$$
	While isotropy subgroups are
	\begin{enumerate}[(i)]
		\item For $F_{e_1, \rho}$ with $\rho> 0$: $H_3\rtimes \left( \{(Id,1)\} \cup \{(S,-1)\}\right)$;
		\item For $F_{e_1,0}$: $H_3\rtimes \left( \{(Id,1)\} \cup \{(-S,1)\}\cup \{(-Id,-1)\} \cup \{(S,-1)\}\right)$;
		\item For $F_{0,1}$ one has $H_3\rtimes \left( \mathrm{SO}(2)\times \{1\} \cup \mathrm{SO}^-(2)\times \{-1\}\right),$
	\end{enumerate}
	where $S:\vv\to \vv$ is the linear map determined by the values $S(e_1)=-e_1, S(e_2)=e_2$.
\end{prop}

To compute the isotropy subgroups, find $B:\hh_3\to \hh_3$ such that $(B,1)\cdot F_{e_1,\rho}=F_{e_1,\rho}$ or $(B,-1)\cdot F_{e_1,\rho}=-F_{e_1,\rho}$. Se more in \cite{OS3}.

\medskip



As above, for a left-invariant Lorentz force $F_{U,\rho}$ we write $U=\beta e_1+\alpha e_2\in\vv $. Since  left translations are symmetries of the magnetic equations for any invariant Lorentz force $F$, it suffices to compute the magnetic geodesics through the identity element.

Let $\exp:\hh_3 \to H_3$ be the exponential map, then we write as in \cite{OS3}, $\gamma(t) = \exp(x(t)e_1+y(t)e_2+z(t)e_3)$ such that $x(0)=y(0)=z(0)=0$ and $x'(0)=x_0,\, y'(0)=y_0,\, z'(0)=z_0$. The magnetic equations corresponding to $F_{U,\rho}$ are
	\begin{equation}\label{magnetic-Heis}
		\left\{ \begin{array}{rcl}
			x''(t)+\left(z'(t)+\frac12(y'(t) x(t) -x'(t) y(t))+\rho\right) (y'(t)+\beta) & = &  \rho \beta \\
			y''(t)-\left(z'(t)+\frac12(y'(t) x(t) -x'(t) y(t))+\rho\right) (x'(t)- \alpha)  & = & \rho \alpha \\
			z''(t)+\frac{1}{2}(x''(t)y(t)-x(t)y''(t)) & = & \beta x'(t)+\alpha y'(t).\\
		\end{array} \right.
	\end{equation}

	By making use of the action of  $\Iso(H_3)\times \RR^*$  on the set of left-invariant Lorentz forces, as in Proposition \ref{orbits}, magnetic geodesics arise by solving one of the following systems with the same initial conditions:  $x(0)=y(0)=z(0)=0$ and $x'(0)=x_0,\, y'(0)=y_0,\, z'(0)=z_0$.
	\begin{enumerate}[(i)]
		\item For $F_{0,\rho}$ (actually it suffices $\rho=1$)	\begin{equation}\label{magnetic-exact}
			\left\{ \begin{array}{rcl}
				x''(t)+(z_0+\rho) y'(t) & = &  0\\
				y''(t)-(z_0+\rho) x'(t)  & = & 0 \\
				z'(t)+\frac{1}{2}(x'(t)y(t)-x(t)y'(t)) & = & z_0.
			\end{array} \right.
		\end{equation}
\item For $F_{e_1,\rho}$ with $\rho\geq 0$:
	\begin{equation}\label{magnetic-type2HeisSimp2}
		\left\{ \begin{array}{rcl}
			x''(t)+( x(t)+z_0+\rho) (y'(t)+1) & = &  \rho\\
			y''(t)-( x(t)+z_0+\rho) x'(t)  & = & 0 \\
			z'(t)+\frac{1}{2}(x'(t)y(t)-x(t)(y'(t)+2)) & = & z_0.
		\end{array} \right.
	\end{equation}
	\end{enumerate}
	
	In \cite{EGM} the authors obtain the explicit solutions for (i) as:
	\begin{itemize}
		\item if $z_0+\rho\neq 0$, the solution is
		$$\left( \begin{matrix}
			x(t)\\
			y(t)
		\end{matrix}\right) = \frac{1}{z_0+\rho}\left( \begin{matrix}
			\sin(t(z_0+\rho)) & -1+\cos(t(z_0+\rho))\\
			1-\cos(t(z_0+\rho)) & \sin(t(z_0+\rho))
		\end{matrix}
		\right)\left( \begin{matrix} x_0\\ y_0 	\end{matrix}
		\right),$$
		and for $V_0=x_0e_1+y_0e_2$ set
		$$z(t)=\left(z_0+\frac{\Vert V_0\Vert^2}{2(z_0+\rho)}\right)t - \frac{\Vert V_0\Vert^2}{2(z_0+\rho)^2}\sin(t(z_0+\rho)).$$
		\item If $z_0=-\rho$ the solution is 
		$\gamma(t)=\exp(t(x_0e_1+y_0 e_2+ z_0 e_3))$. 
	\end{itemize}
	
	For (ii), Lemma 3.1 in \cite{OS3} establishes a one-to-one correspondence between:
	$$\left\{\begin{array}{c}
		\gamma(t)=\exp(x(t)e_1+ y(t) e_2 + z(t)e_3):\\ \mbox{solutions of Equation \eqref{magnetic-type2HeisSimp2}}\\
		\gamma(0)=e \mbox{ and } \gamma'(0)=x_0e_1+y_0e_2+z_0e_3
	\end{array}\right\}  \longleftrightarrow     \left\{\begin{array}{c}
		x(t): \mbox{solution of Equation }\\
		x''(t)+h'(x(t)) h(x(t))=\rho \\
		x(0)=0 \mbox{ and } x'(0)=x_0
	\end{array}\right\} $$ 
	where	$h(x)= \frac{x^2}{2}+(z_0+\rho)x+y_0+1$.

	Moreover, the functions $y(t)$ and $z(t)$ are determined by $x(t)$ as follows
	\begin{equation}\label{eqmagneticHeisy}y(t)=\int_{0}^{t}\left(\frac{x(s)^2}{2}+(z_0+\rho)x(s)+y_0\right)ds,\end{equation} 
	\begin{equation}\label{eqmagneticHeisz}z(t)=-\frac{1}{2}x(t)y(t)-(z_0+\rho)y(t)-x'(t)+x_0.\end{equation}
	while $x(t)$ is solution of the equation
	$$
	x''(t)+h'(x(t)) h(x(t))=\rho, \qquad 
	x(0)=0 \mbox{ and } x'(0)=x_0,
$$ 
where	$h(x)= \frac{x^2}{2}+(z_0+\rho)x+y_0+1$.

Thus, the function $x(t)$ solve the equation:
\begin{equation}\label{Eqonev}
	x'(t)^2 =\Vert V_0+e_2\Vert ^2-h(x(t))^2 + 2\rho x(t).
\end{equation}
We summarize below the main results we shall need later. 	As proved in \cite{OS3} the following polynomial plays a role in the description of solutions 
	\begin{equation}\label{polinQ}
	\begin{array}{rcl}	P(\eta) & = &  
		\Vert V_0+e_2\Vert^2-\left(\frac{\eta^2}{2}+y_0+1-\frac{(z_0+\rho)^2}{2}\right)^2+2\rho (\eta-z_0-\rho)\\
	 & = &  	-\frac{1}{4}\left(\eta^4 + 2p_0 \eta^2 -8\rho \eta + q_0  \right) 
	\end{array}
	\end{equation}  
where
$$
p_0=2(y_0+1) - (z_0+\rho)^2  \ \mbox{ and } \  q_0=p_0^2+8\rho(z_0+\rho) - 4\Vert V_0+e_2\Vert^2.
$$
Let $\Delta$ denote the discriminant of $P$ and $r_1$, $r_2$, $r_3$ and $r_4$ the roots of $P$ in $\CC$. It was proved in \cite{OS3} that
\begin{itemize}
	\item If $\Delta<0$ then 
 \begin{equation}\label{solcase1}
		x(t)=\frac{\left(r_1 \delta_4-r_4 \delta_1\right)\mathrm{cn}\left(\frac{\sqrt{\delta_1 \delta_4}}{2}t + C_1,k\right)+r_1 \delta_4 + r_4 \delta_1 }{\left(\delta_4-\delta_1\right)\mathrm{cn}\left(\frac{\sqrt{\delta_1 \delta_4}}{2}t+C_1,k\right) + \delta_1 + \delta_4} -z_0-\rho,
	\end{equation}
	where $\mathrm{cn}$ is the cosine amplitude, Jacobi's elliptic function, with 
	
	$\delta_1=|r_1-r_2|=\sqrt{2p_0+2r_1^2+(r_1+r_4)^2}$, and 
	$\delta_4=|r_4-r_2|=\sqrt{2p_0+2r_4^2+(r_1+r_4)^2}$,

	$k^2=\frac{(r_4-r_1)^2 - (\delta_4-\delta_1)^2}{4\delta_1 \delta_4}$ and $C_{1}$ is a constant.  In this case $r_1 < r_4\in\RR$ and $r_2 =\overline{r_3}\in\CC-\RR$.  
	
	\item if $\Delta>0$, then 
		\begin{equation}\label{solcase21}	
		x(t)=r_4-\frac{r_4-r_1}{1+\frac{r_2-r_1}{r_4-r_2}\mathrm{sn}^2\left(\frac{\sqrt{(r_4-r_2)(r_3-r_1)}}{4}	t + C_{21},k_1\right)}-z_0-\rho,
	\end{equation}
	where $k_1=\sqrt{\frac{(r_4-r_3)(r_2-r_1)}{(r_4-r_2)(r_3-r_1)}}$ and $C_{21}$ is a constant. Here all the roots are real satisfying $r_1 < r_2 < r_3 < r_4$. 

	\item $\Delta=0$.   Let $r\in\RR$ be the root of $P$ with multiplicity two and consider $\mu=\frac{p_0 +3r^2}{2}=\frac{1}{4}(r-r_2)(r-r_3)$. The roots of $P$ are $r$, $r_2=-r-\sqrt{-2(p_0+r^2)}$   and $r_3=-r+\sqrt{-2(p_0+r^2)}$. Since they are real, we have that $p_0+r^2<0$.
	
	If $\mu=\frac{p_0 +3r^2}{2}>0$ the solution of \eqref{eqmagneticHeisx} is given by 
	\begin{equation}\label{sol3a}
		x(t)=\frac{-2\mu }{r+\sqrt{r^2-\mu}\cos\left(\sqrt{\mu}\,t-C_4\right)}+r-z_0-\rho,
	\end{equation}
	where
	$C_4$ is a constant. 
	
	If $\mu=\frac{p_0 +3r^2}{2}<0$. There are two possibilities, either   $r_2=-r-\sqrt{-2(p_0+r^2)}\leq z_0+\rho<r$ or $r<z_0+\rho\leq r_3= -r+\sqrt{-2(p_0+r^2)}$.
	
	If $r<z_0+\rho\leq -r+\sqrt{-2(p_0+r^2)}$, the solution is
	\begin{equation}\label{sol3b1}
		x(t)=\frac{- 2\mu
		}{r+\sqrt{r^2-\mu} \cosh(\sqrt{-\mu}\,t+C_5)}+r-z_0-\rho,
	\end{equation}
	where
	$C_5$ is a constant. 

	If $-r-\sqrt{-2(p_0+r^2)}\leq z_0+\rho<r$, we have the solution 
	\begin{equation}\label{sol3b2}
		x(t)=\frac{- 2 \mu
		}{r-\sqrt{r^2-\mu} \cosh(\sqrt{-\mu}\,t+C_6)}+r-z_0-\rho,
	\end{equation}
	where
	$C_6$ is a constant. 
	
	

	\smallskip
	
 Finally in the case that $\mu=\frac{p_0 +3r^2}{2}=0$, the solution of Equation \eqref{eqmagneticHeisx} is given by
	\begin{equation}\label{sol3c}
		x(t)=\frac{ -4r}{1+r^2(t+C_7)^2}+r-z_0-\rho,
	\end{equation}
	where $C_7$ is constant.  
\end{itemize}

\begin{rem}\label{remperiod}
	The solution $x(t)$ is periodic whenever $\Delta \neq 0$ or $\Delta = 0$ with $\mu > 0$. The periods are given by
	\begin{itemize}
		\item $\frac{8K(k)}{\sqrt{\delta_1 \delta_2}}$ if $\Delta < 0$,
			\item $\frac{2\pi}{\sqrt{\mu}}$ if $\Delta = 0$ and $\mu > 0$,
		\item $\frac{8K(k)}{\sqrt{(r_4-r_2)(r_3-r_1)}}$ if $\Delta > 0$,
	
	\end{itemize}
	where $K$ denotes the complete elliptic integral of the first kind.
\end{rem}

	\section{Periodic magnetic trajectories on $H_3$}\label{section5}
	
	In this section we focus on closed magnetic trajectories on the Heisenberg Lie group $H_3$. We begin by recalling some definitions.
	
	A curve $\gamma:I\to M$ is said to be {\em non-simple} if there exist $t_1<t_2$ in $I$ such that $\gamma(t_1)=\gamma(t_2)$, while  a curve $\gamma:\RR\to M$ is called {\em periodic} if there exist $\omega>0$ such that $\gamma(t)=\gamma(t+\omega)$ for all $t\in\RR$.

	
	\begin{prop}\label{closedperiodic}
		Let $\gamma:\RR\to M$ be a magnetic trajectory for a left-invariant Lorentz force $F_{e_1, \rho}$ on the Heisenberg Lie group $H_3$. If $\gamma$ is non-simple then it is periodic.
		
	\end{prop}
	\begin{proof}
		We have already seen that magnetic trajectories on $H_3$ are defined for every $t\in\RR$. By performing a group translation and a translation of the parameter, we can assume that the non-simple magnetic curve $\gamma(t)=\exp(x(t)e_1+y(t)e_2+z(t)e_3)$ passes through the identity element and intersects itself there. Thus, for a positive parameter $\omega$, we have $x(\omega) = y(\omega) = z(\omega) = 0$. 
		
		As already explained the magnetic trajectory of $F_{e_1,\rho}$ depends on the solution of
			\begin{equation}\label{eqmagneticHeisx}
			x''(t)+( x(t)+z_0+\rho) \left(\frac{ x(t)^2}{2}+(z_0+\rho) x(t)+y_0+1\right) =  \rho.
		\end{equation} Verify that evaluating Equation  \eqref{eqmagneticHeisz} at  $t=\omega$ we get $x'(\omega)=x_0$. Consequently, both $x(t)$ and $x(t+\omega)$ are solutions of \eqref{eqmagneticHeisx} with the same initial conditions, and therefore $x(t)=x(t+\omega)$ for all $t$. Then Equation \eqref{eqmagneticHeisy} implies that $y'(t)$ is periodic with period $\omega$, which gives $y(t+\omega)-y(t)=cte. =y(\omega)-y(0)=0$ for all $t$. Finally, since both $x(t)$ and $y(t)$ are periodic of period $\omega$ we get from \eqref{eqmagneticHeisz} that $z(t)$ is also periodic with the same period. 	
	\end{proof}

	In view of the result above, from now on we shall investigate the existence of periodic magnetic trajectories. For us, these curves are the closed ones. 
	
	Recall that, by Lemma \ref{lem2}, two Lorentz forces lying in the same orbit under the action of $\Iso(M)\times\RR^*$ (see  \eqref{action1}) determine equivalent systems and the isotropy subgroup of a fixed Lorentz force gives a symmetry group of the corresponding magnetic equation. Moreover, the induced action on curves transforms periodic magnetic trajectories into periodic ones. 
	
	The classes of Lorentz forces were determined in Proposition \ref{orbits} so as the isotropy subgroups.  We showed that $F_{0,1}$ and $F_{e_1,\rho}$ for $\rho\geq 0$ are representatives of each orbit and that the isotropy subgroup of each Lorentz force $F$ contains the group $H_3\rtimes \left( \{(Id,1)\} \cup \{(S,-1)\}\right)$. Then we can focus, for each one of the mentioned Lorentz forces, on the magnetic trajectories that verify $\gamma(0)=e$ and $\gamma'(0)=x_0e_1+y_0e_2+z_0e_3$ with $x_0\geq0$.

	The orbit of $F_{0,1}$ corresponds to left-invariant exact Lorentz forces. In Theorem 2 of \cite{EGM} it is proved that in this case there are periodic magnetic trajectories only on low energy levels. Observe that the definitions of energy and of the coordinate $z_0$ given there, are slightly different from ours. More precisely, we can rewrite their theorem with our notation:
	\begin{thm}\label{periodicexact}
		\cite{EGM} For the Lorentz force $F_{0,\rho}$, there exist periodic magnetic trajectories with energy $\mathcal E$ if and only $0<\mathcal E<\frac{\rho^2}{2}$. For any $0<\mathcal E<\frac{\rho^2}{2}$, let $z_0 = -\rho+\sgn(\rho) \sqrt{\rho^2-2 \mathcal E}$ and let $x_0$ and $y_0$ be any numbers such that $x_0^2+y_0^2=2 \mathcal E-z_0^2$. Then, the magnetic trajectory with initial condition $\gamma(0)=e$ and $\gamma'(0)=x_0e_1+y_0e_2+z_0e_3$ will be periodic with  energy $\mathcal E$.
	\end{thm} 
	
	Thus, here we want to study periodic magnetic trajectories for the Lorentz forces $F_{e_1,\rho}$ with $\rho\geq0$. To characterize them we have the next lemma. 
	
	\begin{lem}\label{Propperiodic}
		A magnetic trajectory $\gamma(t)=\exp(x(t)e_1+y(t)e_2+z(t)e_3)$ for $F_{e_1,\rho}$ through the identity is periodic if and only if $x(t)$ is periodic of period $\omega$ and $y(\omega)=0$.
	\end{lem}
	
	\begin{proof}
		$\Rightarrow)$ Trivial. 
		
		$\Leftarrow)$ The periodicity of $x(t)$ implies that $x'(t)$ is also periodic and $x'(\omega)=x'(0)=x_0$. By evaluating Equation \eqref{eqmagneticHeisz} in the period  $\omega$, one gets that $z(\omega)=0$. Thus, it holds   $x(\omega)=y(\omega)=z(\omega)=0$ and from Proposition \ref{closedperiodic} the trajectory is periodic.
	\end{proof}
	
	In summary, to check periodicity, it suffices to verify that $y(\omega)=0$ whenever $x(t)$ is periodic of period $\omega$. In the previous section we recalled the solutions obtained in \cite{OS3}. We will use them to proved the following proposition.

\begin{prop}
	Let $\rho\geq0$. For the Lorentz force  $F_{e_1,\rho}$ on the Heisenberg Lie group, there are no periodic magnetic trajectories corresponding to $\Delta\geq0$. 
\end{prop}
\begin{proof}
	From  Lemma \ref{Propperiodic} we must see when
	\begin{equation}\label{y(w)}
		y(\omega)=\displaystyle \int \limits_{0}^{\omega}  \left(\frac{x(t)^2}{2}+(z_0+\rho)x(t)+y_0\right)dt,
	\end{equation}
	is zero in the cases where $x(t)$ is periodic of period $\omega$.

	If $\Delta >0$ and 	 $r_1 \leq z_0+\rho \leq r_2$, the solution $x(t)$ is given by \eqref{solcase21} and it is periodic of period $\omega=\frac{8K(k)}{\sqrt{(r_4-r_2)(r_3-r_1)}}$.
	
	As in the Appendix, we denote by $K(k)$ and $E(k)$  the complete elliptic integral of the first and second kind, respectively, while $\Pi(\alpha^2,k)$ is the complete elliptic integral of the third kind.

	 Next, we substitute into Equation \eqref{y(w)} using the change of variable $s=\frac{\sqrt{(r_4-r_2)(r_3-r_1)}}{4}\,t + C_{21}$, together with the fact that $\mathrm{sn}^2(s, k_1)$ has period $2K(k)$ and the integrals \eqref{Equ46} and \eqref{Equ47} in the Appendix, to obtain 
	\begin{align*}
		y(\omega) &=\frac{4}{\sqrt{(r_4-r_2)(r_3-r_1)}}\int \limits_{0}^{2K(k)}  \left(\frac{1}{2}\frac{(r_4-r_1)^2}{\left(1+\frac{r_2-r_1}{r_4-r_2} \mathrm{sn}^2(s, k_1)\right)^2}+\frac{r_4(r_1-r_4)}{1+\frac{r_2-r_1}{r_4-r_2} \mathrm{sn}^2(s, k_1)}\right.\\
		&\left.+\frac{1}{2}r_4^2-\frac{1}{2}(z_0+\rho)^2+y_0\right)ds\\
		&=\frac{2(r_4-r_1)}{\sqrt{(r_4-r_2)(r_3-r_1)}}\left(4r_4\Pi(-A,k_1)+\frac{(r_4-r_2)(r_3-r_1)}{r_4-r_1}E(k_1)-(r_4-r_2)K(k_1)\right)\\
		&-8\frac{r_4(r_4-r_1)}{\sqrt{(r_4-r_2)(r_3-r_1)}}\Pi(-A,k_1) +\frac{4 (r_4^2+2y_0-(z_0+\rho)^2)}{\sqrt{(r_4-r_2)(r_3-r_1)}} K(k_1)\\
		&=2\sqrt{(r_4-r_2)(r_3-r_1)} \left(E(k_1)-\frac{r_1 r_2 +r_3 r_4-4y_0+2(z_0+\rho)^2}{(r_4-r_2)(r_3-r_1)} K(k_1)\right)\\
		&=2\sqrt{(r_4-r_2)(r_3-r_1)} \left(E(k_1)-K(k_1)-\frac{4+(r_2+r_3)^2}{(r_4-r_2)(r_3-r_1)} K(k_1)\right).
	\end{align*}
 	Since $0<E(k_1)<K(k_1)$ for $k_1>0$ we conclude that $y(\omega)<0$. 
	
	For $\Delta>0$ and $r_3 \leq z_0+\rho \leq r_4$, we arrive analogously to the same formula for $y(\omega)$  so it is also non-zero. 
	
	Now consider $\Delta=0$. The only case where $x(t)$ is periodic is when $\mu>0$. Moreover $x(t)$ is given by Equation  \eqref{sol3a} and it has period $\omega=\frac{2\pi}{\sqrt{\mu}}$. We use the integrals in Equation (50) in the Appendix to compute
	\begin{align*}
		y(\omega) =&=\frac{1}{\sqrt{\mu}}\int \limits_{0}^{2\pi}  \left(\frac{2\mu^2}{\left(r+\sqrt{r^2-\mu}\cos(s)\right)^2}-\frac{2\mu r}{r+\sqrt{r^2-\mu}\cos(s)}+\frac{1}{2}r^2-\frac{1}{2}(z_0+\rho)^2+y_0\right)ds\\
		&= \frac{1}{\sqrt{\mu}}(p_0+r^2-2)\pi
	\end{align*}
	Recall that the observation above Equation \eqref{sol3a} shows that $p_0+r^2<0$, so $y(\omega)$ is negative too. This proves the non-existence of periodic trajectories for $\Delta\geq0$.
\end{proof}

For $\Delta<0$, the magnetic trajectories have $x(t)$ periodic. Hence, applying Lemma \ref{Propperiodic} as before, it suffices to determine when $y(\omega)$ in Equation \eqref{y(w)} vanishes, where $\omega=\frac{8K(k)}{\sqrt{\delta_1\delta_4}}$ is the period of $x(t)$.
Substituting $x(t)$ from Equation \eqref{solcase1} and using the change of variable $s=\frac{\sqrt{\delta_1\delta_4}}{2}t+C_1$, we obtain (see Equations \eqref{int1cn} and \eqref{int1cn2} in the Appendix)
	
			\begin{align*}	y(\omega)  = & \int \limits_{0}^{\omega}  \left(\frac{x(t)^2}{2}+(z_0+\rho)x(t)+y_0\right)dt\\
				=&\frac{(\delta_1+\delta_4)(\delta_1^2-\delta_4^2-2r_1^2+2r_4^2)}{\sqrt{\delta_1 \delta_4}(\delta_4-\delta_1)}\,\Pi\left(-\frac{(\delta_4-\delta_1)^2}{4\delta_1 \delta_4}, k\right) + 4 \sqrt{\delta_1 \delta_4} E(k)\\
				&  +\frac{4}{\sqrt{\delta_1 \delta_4}}\left(\frac{\delta_4 r_1^2 - \delta_1 r_4^2}{\delta_4-\delta_1} +2y_0 - (z_0+\rho)^2\right) K(k).
			\end{align*}
			We can simplify this expression by recalling that
			
			$\delta_1=\sqrt{2p_0+2r_1^2+(r_1+r_4)^2}$ and $\delta_4=\sqrt{2p_0+2r_2^2+(r_1+r_4)^2}$, from which it follows that
			\begin{equation}\label{eqdifdistanc}
				\delta_1^2-\delta_4^2=2r_1^2 - 2r_4^2.
			\end{equation}	
			Thus the coefficient corresponding to the function $\Pi$ vanishes. Furthermore, we can use Equation \eqref{eqdifdistanc} to rewrite:
			
			\begin{align*}
				2\frac{\delta_4 r_1^2 - \delta_1 r_4^2}{\delta_4-\delta_1}  = & \frac{(r_1^2+r_4^2)(\delta_4-\delta_1)-(r_4^2-r_1^2)(\delta_1+\delta_4)}{\delta_4-\delta_1}\\  = & \frac{(r_1^2+r_4^2)(\delta_4-\delta_1)-\frac{1}{2}(\delta_4^2-\delta_1^2)(\delta_1+\delta_4)}{\delta_4-\delta_1}\\
				= & (r_1^2+r_4^2)-\frac{1}{2}(\delta_1+\delta_4)^2.
			\end{align*}
			
			On the other hand, from the definition of $\delta_1$ and $\delta_2$ we get
			\begin{equation*}\label{eqVieta}
				2y_0 - (z_0+\rho)^2 = p_0-2 = \frac{\delta_1^2+\delta_4^2}{4}-r_1^2-r_4^2-r_1r_4 -2,
			\end{equation*}
			and we conclude that:
			\begin{equation}\label{ywfunc1}
				y(\omega)=4 \sqrt{\delta_1 \delta_4} \left(E(k) -\frac{(r_1+r_4)^2+\delta_1\delta_4+4}{2\delta_1 \delta_4} K(k)\right).
			\end{equation}

			In order to work with this last expression it will be useful to reparametrize the set of initial conditions $x_0 e_1+y_0 e_2 +z_0 e_3 \in \hh_3$ with $x_0\geq0$ corresponding to $\Delta<0$, by using triples $(c,d,e)$ in some proper subset of $\mathbb R^3$. 
			
			Take $(c,d,e)\in(0,\infty)\times (0,1)\times [-1,1]$ defined in the following manner:
			$$c=\frac{\sqrt{2p_0+r_1^2+r_4^2}}{2} \qquad d =  \frac{r_4-r_1}{2\sqrt{2p_0+r_1^2+r_4^2}} \qquad  e =\frac{2}{r_4-r_1}\left(z_0+\rho-\frac{r_1+r_4}{2}\right),$$
			where $r_1 < r_4$ are the distinct real roots of the polynomial $P$, defined in Equation   \eqref{polinQ}. Since $\Delta<0$, $2p_0+r_1^2+r_4^2> 0$ and $r_1\leq z_0+\rho \leq r_4$, these variables are well-defined within their respective intervals.  Reciprocally, by considering 
			Vieta's formulas:
			\begin{eqnarray}
				r_2+r_3&=&-(r_1+r_4)\label{Viete1}\\
				r_2 r_3&=& 2p_0 + r_1^2 + r_4^2 + r_1 r_4\label{Viete2}\\
				(r_1+r_4)(2p_0 + r_1^2 + r_4^2) &=& 8\rho, \label{Viete3} 	
			\end{eqnarray} 
			 we obtain the initial conditions 
			\begin{equation}\label{condinitcd}
				\begin{array}{rcl}
					x_0 &=&\frac{2d}{c}\sqrt{1-e^2}\sqrt{(c^3de+\rho)^2+(1-d^2)c^6},   \\
					y_0& =& c^2(2d^2e^2-2d^2+1)+\frac{2de \rho}{c}-1,\\
					z_0&=& \frac{\rho}{c^2} + 2cde-\rho.
				\end{array}
			\end{equation}
			Furthermore, the roots of $P$ (Equation \eqref{polinQ}) obey the equations
			\begin{align}\label{rootscde}
				r_1 = \frac{\rho}{c^2} - 2cd  \qquad   r_4 = \frac{\rho}{c^2} + 2cd \qquad r_2 = -\frac{\rho}{c^2} - 2c\sqrt{1-d^2} i \qquad r_3 = -\frac{\rho}{c^2} + 2c\sqrt{1-d^2} i,
			\end{align}  
			while for $\delta_1=|r_1-r_2|$ and $\delta_4=|r_4-r_2|$ it holds
			$$
			\delta_1=\frac{2}{c^2}\sqrt{\rho^2+c^6-2\rho d c^3}  \qquad \delta_4=\frac{2}{c^2}\sqrt{\rho^2+c^6+2\rho d c^3}.
			$$

			Observe that the energy of the magnetic trajectory, as defined in \eqref{energy}, can be expressed in terms of $(c,d,e)$ by the formula
			
			\begin{equation}\label{Energycde}
				\En(c,d,e)=\frac{(c^4+\rho^2)(c^4+4c^2 d^2-2c^2+1)}{2c^4}.
			\end{equation}
			Note that this expression is independent of $e$.
			
			Finally, if $\gamma(t)=\exp(x(t)e_1+y(t)e_2+z(t)e_3)$ is the magnetic trajectory associated to $(c,d,e)$ and
			$\omega$ is the period of $x(t)$, then the value $y(\omega)$ given in Equation \eqref{ywfunc1} satisfies  $y(\omega)=\Psi(c,d,e)$, for $\Psi$ the function
			\begin{equation}\label{ywfunc2}\small
				\Psi(c,d,e)=\frac{8}{c^2} \sqrt[4]{(\rho^2+c^6)^2-4\rho^2 d^2 c^6} \left(E(k) -\left(\frac{\rho^2+c^4}{2\sqrt{(\rho^2+c^6)^2-4\rho^2 d^2 c^6}} +\frac{1}{2}\right) K(k)\right),
			\end{equation} 
			where $k=\sqrt{\frac{2c^6 d^2 -\rho^2-c^6}{2 \sqrt{(\rho^2+c^6)^2-4\rho^2 d^2 c^6}}+\frac{1}{2}}$. The function $\Psi$ is defined on $(0,\infty)\times (-1,1) \times \RR$.

			\begin{lem}\label{unicityofperiodic}
				For any $(c,e)\in(1,\infty)\times[-1,1]$ there exists a unique $d=d_c\in (0,1)$ such that the magnetic trajectory $\gamma(t)$ through the identity with initial condition $\gamma'(0)=x_0 e_1+y_0 e_2 +z_0 e_3$, as in Equation \eqref{condinitcd},  is periodic.
				
				Furthermore,  these are all possible periodic magnetic trajectories through the identity with $x_0\geq0$.   
			\end{lem}
			
			\begin{proof}

				
				As we already said, in order to get a periodic magnetic trajectory through the identity element we need to determine when the function  $\Psi(c,d,e)$ in Equation \eqref{ywfunc2} vanishes, by Lemma \ref{Propperiodic}. It is immediate to see that we just need to study the roots of the function
				\begin{equation}
					\widetilde{\Psi}(c,d)= E(k) -\left(\frac{\rho^2+c^4}{2\sqrt{(\rho^2+c^6)^2-4\rho^2 d^2 c^6}} +\frac{1}{2}\right) K(k)
				\end{equation}
				for $(c,d)\in(0,\infty)\times(0,1)$.
				
				Since $E(k)< K(k)$ for any $k>0$, if  $\frac{\rho^2+c^4}{\sqrt{(\rho^2+c^6)^2-4\rho^2 d^2 c^6}}\geq 1$, then $\widetilde{\Psi}(c,d)<0$. Notice that the map $\frac{\rho^2+c^4}{\sqrt{(\rho^2+c^6)^2-4\rho^2 d^2 c^6}}$ is increasing as a function of $d\in(0,1)$. 
				
				For a fixed  $c$, assume that there exists $\underline{d}$  such that  $\widetilde{\Psi}(c,\underline{d})=0$. Thus, 
				\begin{equation}\label{dmonotony}
					\frac{\rho^2+c^4}{\sqrt{(\rho^2+c^6)^2-4\rho^2 d^2 c^6}}\geq1\implies \underline{d}< d.
				\end{equation}
				
				For $c\leq 1$ one has that $\frac{\rho^2+c^4}{\sqrt{(\rho^2+c^6)^2-4\rho^2 d^2 c^6}}\geq1$ for any $d$, so $\widetilde{\Psi}(c,d)<0$. 
				
				Fixed $c>1$ and verify that
				$$\lim\limits_{d\to 0}\widetilde{\Psi}(c,d)= \frac{c^4(c^2-1)\pi}{4(\rho^2+c^6)}>0, \qquad 0>\lim\limits_{d\to 1^{-}}\widetilde{\Psi}(c,d)=\left\{\begin{array}{cll}
					- \frac{c^4 (c^2+1)\pi}{4(\rho^2-c^6)} & \text{ if } & c < \sqrt[3]{\rho}\\
					-\infty & \text{ if } & c\geq\sqrt[3]{\rho}
				\end{array}.\right.
				$$
				So, by continuity, there exists some $\underline{d}\in(0,1)$ such that $\widetilde{\Psi}(c,\underline{d})=0$.

				Now, we will see that $\frac{\partial\widetilde{\Psi}}{\partial d}(c,\underline{d})<0$ whenever $(c,\underline{d})$ is a root of $\widetilde{\Psi}$. This will imply that there exists exactly one $\underline{d}$ for each $c>1$. 
				
				Consider the auxiliary function 
				\begin{align*}
					\Psi_1(d)&=\frac{\partial\widetilde{\Psi}}{\partial d}(c,d)- \frac{c^4 (c^6 -2 d^2\rho^2 +\rho^2)(2c^2d^2-c^2+1)}{d(1-d^2)((c^6+\rho^2)^2-4c^6 d^2 \rho^2)}\widetilde{\Psi}(c,d)\\
					&=\frac{c^6(A_1d^4+A_2d^2+A_3)}{4d(1-d^2)((\rho^2+c^6)^2-4\rho^2 d^2 c^6)^{3/2}}K(k)
				\end{align*}     	
				where the coefficients $A_i$ are
				\begin{itemize}
					\item $A_1=8\rho^2(2c^4+\rho^2)$, \item $A_2=-4c^{10}+2c^6\rho^2-16c^4\rho^2+2c^2\rho^2-8\rho^4$ and  
					\item $A_3=-c^2(c^2-1)^2(c^6+\rho^2)$.
				\end{itemize} 
				
				If $\rho=0$, then $A_1=0$, $A_2<0$ and $A_3<0$, so $\Psi_1(d)<0$ for every $d$ and if $\underline{d}>0$ is such that $\Psi (c,\underline{d})=0$ then $\frac{\partial\widetilde{\Psi}}{\partial d}(c,\underline{d})=\Psi_1(\underline{d})<0$
				
				In the case $\rho\neq0$, $A_1>0$ and $A_3<0$ so the quadratic polynomial 
				$$q(x)=A_1d^2+A_2d+A_3$$
				has a positive and a negative root that we denote by $d_1^2$ and $-d_2^2$, respectively, with $d_1,\,d_2>0$. Since the elliptic integral $K(k)$ never vanishes, the real roots of $\Psi_1$ are exactly $d_1$ and $-d_1$. In particular,  $\Psi_1$ has at most one root in the open interval $(0,1)$.
				
				Denote by 
				$\nu:=(\rho^2+c^4)^2-(\rho^2+c^6)^2+4\rho^2d_1^2c^6$, $\zeta:= (\rho^2+c^4)^2-(\rho^2+c^6)^2-4\rho^2d_2^2c^6$.  Using Vieta's relations between the roots of the quadratic $q$ given by
				\begin{equation*}
					d_2^2 - d_1^2=\frac{A_2}{A_1}, \ \ \ \ d_1^2 d_2^2 = -\frac{A_3}{A_1},
				\end{equation*}
				
				we have 
				\begin{equation*}
					\nu \zeta=-\frac{2c^8(c^4-1)(c^4+2\rho^2)(c^4+\rho^2)^2}{2c^4+\rho^2}<0.
				\end{equation*}\normalsize
				Since $\nu > \zeta$, we get that $\nu>0$ which is equivalent to
				$\frac{\rho^2+c^4}{\sqrt{(\rho^2+c^6)^2-4\rho^2 d_1^2 c^6}}>1$.
				So for any $\underline{d}$ such that $\widetilde{\Psi} (c,\underline{d})=0$ we have $0<\underline{d}<d_1$ by Equation  \eqref{dmonotony}. Therefore  $q(\underline{d}^2)<0$ which implies $\Psi_1(\underline{d})<0$ and finally  $\frac{\partial\widetilde{\Psi}}{\partial d}(c,\underline{d})<0$. 
				
				We conclude that for every $c>1$ there exists only one such $\underline{d}$ that we denote by $d_c$ satisfying 
				\begin{itemize}
					\item $0<d_c<1$, 
					\item $\Psi(c,d_c)=0$.
				\end{itemize}  Then  $(c,d_c,e)$ determines a periodic magnetic trajectory for any $c>1$ and $e\in[-1,1]$.  
				
			\end{proof}

			\begin{rem}\label{remposPsi}
				Observe that from the proof we have that if $d<d_c$ then $\Psi(c,d,e)>0$ and if $d>d_c$ then $\Psi(c,d,e)<0$.
			\end{rem}

			From the previous lemma we have that there exist an infinite number of periodic magnetic trajectories with $\Delta<0$ through the identity with $x_0\geq 0$. Now we are going to characterize their energy level.

			\begin{thm}\label{thenergyperiodic}
				Let $F_{e_1,\rho}$ with $\rho\geq0$ denote a Lorentz force on the Heisenberg Lie group. For every energy $\mathcal E>0$ there are an infinite number of periodic magnetic trajectories with energy $\mathcal E$ through the identity with $x_0\geq0$.
				
				 More precisely, there exists exactly one $c>1$ such that the magnetic periodic curves associated to $(c,d_c, e)$ for  $e\in[-1,1]$ (Lemma \ref{unicityofperiodic}),  have energy $\mathcal E$. 
			\end{thm}
			\begin{proof}
				Remember the expression of the energy in terms of $(c,d,e)$ given in Equation \eqref{Energycde}, which  for periodic magnetic trajectories only depends on $c>1$:
				\begin{equation}\label{energyf}
					\En(c)=\frac{(c^4+\rho^2)(c^4+4c^2 d_c^2-2c^2+1)}{2c^4}
				\end{equation}	
				where $d(c)=d_c$ is the function defined by the condition $\widetilde{\Psi}(c,d_c)=0$ for every $c>1$. And it was proved in the previous lemma that $\frac{\partial\widetilde{\Psi}}{\partial d}(c,d_c)<0$. By the Implicit Function Theorem $d(c)$  is a $C^1$-function on $(1,\infty)$ and $d'(c) =- \frac{\partial\widetilde{\Psi}}{\partial c}(c,d_c) / \frac{\partial\widetilde{\Psi}}{\partial d}(c,d_c)$.  
				
				Let see that $\frac{\partial\widetilde{\Psi}}{\partial c}(c,d_c)>0$ for every $c>1$. As before we use a convenient auxiliary function given by:
				\begin{align*}
					\Psi_2(d)&=\frac{\partial\widetilde{\Psi}}{\partial c}(c,d)- \frac{3c^3\rho^2(c^2-2c^2d^2-1)}{(c^6+\rho^2)^2-4\rho^2d^2c^6}\Psi(c,d)=\frac{c^3(B_1d^2+B_2)}{((\rho^2+c^6)^2-4\rho^2 d^2 c^6)^{3/2}}K(k)
				\end{align*}     
				where the coefficients $B_i$ are
				\begin{itemize}
					\item $B_1=-2c^2\rho^2(2c^4+3\rho^2)$ and 
					\item $B_2=c^{12} + (\frac{3}{2}c^8 + 2c^6 - \frac{3}{2}c^4)\rho^2 + (3c^2-2)\rho^4$.
				\end{itemize} 
				
				If $\rho=0$ then $\Psi_2(d)=\frac{K(k)}{c^3}>0$ for any $d\in(0,1)$.	
				
				If $\rho\neq 0$, then observe that $B_1<0$ and $B_2>0$ since $c>1$, which implies  $-\frac{B_2}{B_1}>0$.

				An application of Equation  \eqref{dmonotony} for $d=\sqrt{-\frac{B_2}{B_1}}$ gives 
				\begin{equation*}
					\frac{\rho^2+c^4}{\sqrt{(\rho^2+c^6)^2+4\rho^2 \frac{B_2}{B_1} c^6}}=\sqrt{\frac{2c^4+3\rho^2}{3\rho^2}}>1\implies d_c< \sqrt{-\frac{B_2}{B_1}}\implies B_1 d_c^2+B_2 >0.
				\end{equation*}
				In any case, we conclude that $\frac{\partial\widetilde{\Psi}}{\partial c}(c,d_c) =\Psi_2(d_c)>0$. Since we have proved in the previous lemma that $\frac{\partial\widetilde{\Psi}}{\partial d}(c,d_c)<0$, we have that $d'(c) >0$ for every $c>1$.
				
				To analyze the behavior  of the function $\En(c)$ we compute
				\begin{equation}\label{derivenergy}
					\En'(c)=\frac{2(D_1 d_c^2 + D_2)}{c^5}+\frac{4(c^4+\rho^2)}{c^2}d_c d'(c)
				\end{equation}
				where $D_1=2c^2(c^4-\rho^2)$ and $D_2=(c^2-1)(c^6+\rho^2)>0$. 
				
				Suppose $D_1<-D_2 <0$, in particular $c^4<\rho^2$. We can apply again  Equation \eqref{dmonotony} for $d=\sqrt{-\frac{D_2}{D_1}} <1$ to get
				\begin{equation*}
					\frac{\rho^2+c^4}{\sqrt{(\rho^2+c^6)^2+4\rho^2 \frac{D_2}{D_1} c^6}}=\sqrt{\frac{\rho^4-c^8}{\rho^4-c^{12}}}>1\implies d_c< \sqrt{-\frac{D_2}{D_1}}\implies D_1 d_c^2+D_2 >0.
				\end{equation*}
				
				If $-D_2\leq D_1<0$ we also have that $D_1 d_c^2+D_2 >0$ since $d_c^2<1$ and the case $D_1\geq0$ is trivial.   
				
				Thus, since $d'(c)>0$ and $D_1 d_c^2+D_2 >0$ we conclude from Equation \eqref{derivenergy} that $\En'(c)>0$ for all $c>1$. Checking that $lim_{c\to1}\En(c)=0$ and $lim_{c\to\infty}\En(c)=\infty$, we can affirm that the  energy function $\En$ gives a bijection between $(1,\infty)$ and $(0,\infty)$. Then for any $\mathcal E>0$ there exists a unique $c>1$ such that  the periodic curves associated to $(c,d_c,e)$ for $e\in[-1,1]$ have energy $\mathcal E$.
						
			\end{proof}		
			
						\begin{example} In particular, we can analyze the case $\rho=0$ with more detail. In this instance we have that $r_4 = - r_1 = \sqrt{z_0^2 + 2\Vert V_0+e_2\Vert - 2(y_0+1)}$, 
				$\delta_1=\delta_4= 2\sqrt{\Vert V_0+e_2\Vert}$. Then the variables $c,d,e$ are given by $c=\sqrt{\Vert V_0+e_2\Vert}$, $d = \frac{\sqrt{z_0^2 + 2\Vert V_0+e_2\Vert - 2(y_0+1)}}{2\sqrt{\Vert V_0+e_2\Vert}}$, $e=\frac{z_0}{\sqrt{z_0^2 + 2\Vert V_0+e_2\Vert - 2(y_0+1)}}$. 
				
				So if $\Vert V_0+e_2\Vert\leq 1$ then the magnetic trajectory is not periodic. On the other hand, since we have that
				$$y(\omega)= 8c\left(E(d) -\left(\frac{1}{2c^2} +\frac{1}{2}\right) K(d)\right)$$
				then  $d_c$ is defined implicitly by $\frac{E(d_c)}{K(d_c)}=\frac{1}{2c^2} +\frac{1}{2}$ for each $\sqrt{\Vert V_0+e_2\Vert}=c > 1$. Then for any $ c> 1$ we have a family of magnetic periodic curves with energy $\En= \frac{c^4+4c^2 d_c^2-2c^2+1}{2}$ and $\sqrt{\Vert V_0+e_2\Vert}=c$.
			\end{example}
			
			\begin{rem} In view of the preceding results, one has the following statement. 
				On the Heisenberg group $(H_3, \la\,,\,\ra)$ equipped with the left-invariant metric and a left-invariant Lorentz force $F$:
				\begin{enumerate}
					\item For $F\equiv F_{0,\rho}$ left-invariant exact, there exist periodic magnetic trajectories with energy $\mathcal E$ if and only  $\mathcal E<\frac{\rho^2}{2}$.
					\item For $F\equiv F_{e_1,\rho}$ with $\rho\geq0$, there exist periodic magnetic trajectories with any energy $\En=\mathcal E>0$.
				\end{enumerate}
				
			\end{rem}

			Until now we consider the action of the symmetry group $H_F\subseteq \Iso(H_3)\times \RR$ to study the solutions of the corresponding magnetic equation. But  there exists an  action of $\RR$ on the set of differentiable curves, $\mathcal C$, given by translations of the parameter 
			\begin{center} $(s\cdot\gamma)(t)=\gamma(t+s)$ \quad for $s\in\RR$.  \end{center}
			Observe that this action preserves the solutions in any autonomous differential system, becoming a symmetry of the differential equations. 
			
			Consider the action of $H_F \ltimes \RR$ on $\mathcal C$ that preserves the magnetic curves for the Lorentz force $F$ given by:
			\begin{equation*}
				((g,s)\cdot \gamma)(t)=(g\cdot\gamma)(t+s), \quad \mbox{for } g \in H_F,  s\in \RR, \gamma \in \mathcal C. 
			\end{equation*} 
			Recall that $g\in H_F$ corresponds to a pair $(\psi,r)\in H_F\subseteq \Iso(H_3)\times \RR$ as  in Lemma \ref{lem2}. 
			
			From Remark \ref{presenergy} this action also preserves the energy of the magnetic trajectories. Thus the transformed of a periodic magnetic trajectory is also a periodic magnetic trajectory with the same energy. Next theorem proves a converse of this fact for any left-invariant Lorentz force in $H_3$.

			\begin{thm} \label{Equienergy}
				Given a left-invariant Lorentz force $F$ in the Heisenberg group $H_3$, any two periodic magnetic trajectories with the same energy are equivalent by the action of the product group $H_F\ltimes\RR$.
			\end{thm}
			
			\begin{proof}
				First consider $F=F_{0,\rho}$. From Theorem \ref{periodicexact},  we have that two periodic magnetic trajectories through the identity with the same 
				energy $\mathcal E$ have initial conditions $(x_0,y_0,z_0)$ and $(\tilde{x}_0,\tilde{y}_0,z_0)$ with $x_0^2+y_0^2=\tilde{x}_0^2+\tilde{y}_0^2 = 2\mathcal{E}-z_0^2$ and $z_0=-\rho+\sgn(\rho) \sqrt{\rho^2-2\mathcal{E}}$. Then, there exist $A\in SO(2)$ such that $A(x_0, y_0)=(\tilde{x}_0, \tilde{y}_0)$ and the corresponding $((A,1),0)\in H_F\ltimes\RR$ transforms one trajectory to the other. Furthermore, if the periodic trajectories have initial condition different from the identity we apply a translation by an element of $H_3$ which is also in $H_F\ltimes\RR$.

				By the observation following Proposition \ref{closedperiodic} about the action on Lorentz forces, we only have to consider $F=F_{e_1,\rho}$ with $\rho\geq0$. Let $\sigma_1(t)=\exp(x_1(t)e_1+y_1(t)e_2+z_1(t)e_3)$ and $\sigma_2=\exp(x_2(t)e_1+y_2(t)e_2+z_2(t)e_3)$ be two periodic curves with the same energy through the identity such that $x_1'(0)=x_0\geq 0$ and $x_2'(0)=\tilde{x}_0\geq 0$. By Theorem \ref{thenergyperiodic} they are associated to $(c,d_c,\tilde{e}_1)$ and $(c,d_c,\tilde{e}_2)$, respectively, since they have the same energy. Denote the corresponding initial conditions given in \eqref{condinitcd} by $(x_0,y_0,z_0)$ and $(\tilde{x}_0,\tilde{y}_0,\tilde{z}_0)$. 
				Since the relations \eqref{rootscde} do not depend on the coordinate $e$, the two solutions have the same associated roots i.e. determine the same polynomial $P(\eta)$ in \eqref{polinQ}. So by remembering the explicit form of a magnetic trajectory from Equation \eqref{solcase1} we get that 
				\begin{equation}\label{relx1x2}
					x_2(t)=x_1(t+C)+z_0-\tilde{z}_0
				\end{equation}
				for all $t$, where  $C$ is a constant such that $x_1(C)=\tilde{z}_0-z_0$. Then, from Equation \eqref{eqmagneticHeisy}  one has 
				\begin{align*}
					y'_2(t)&=\frac{x_2(t)^2}{2}+(\tilde{z_0}+\rho)x_2(t)+\tilde{y_0}\\
					&=\frac{(x_1(t+C)+z_0-\tilde{z}_0)^2}{2}+(\tilde{z_0}+\rho)(x_1(t+C)+z_0-\tilde{z}_0)+\tilde{y_0}\\
					&=y'_1(t+C)+(\frac{z_0^2}{2}+\rho z_0-y_0)-(\frac{\tilde{z}_0^2}{2}+\rho \tilde{z}_0-\tilde{y}_0)
				\end{align*}
				But in terms of the variables $(c,d_c,e)$ (see \eqref{condinitcd}) we have that\\ $\frac{z_0^2}{2}+\rho z_0-y_0 = \frac{(4d^2 - 2)c^6 + (-\rho^2 + 2)c^4 + \rho^2}{2c^4}$ does not depends on $e$, so $y'_2(t)=y_1'(t+C)$. Therefore 	for every $t$ we get
				\begin{equation}\label{rely1y2}
					y_2(t)=y_1(t+C)-y_1(C).
				\end{equation}

				Now, use Equation \eqref{eqmagneticHeisz} to get:
				\begin{align*}
					z_2(t) = & -\frac{1}{2}x_2(t)y_2(t)-(\tilde{z}_0+\rho)y_2(t)-x_2'(t)+\tilde{x}_0\\
					=&-\frac{1}{2}(x_1(t+C)-x_1(C))(y_1(t+C)-y_1(C))\\
					&-(\tilde{z}_0+\rho)(y_1(t+C)-y_1(C))-x_1'(t+C)+\tilde{x}_0\\
					=& z_1(t+C)-z_1(C)+(z_0-\tilde{z}_0)(y_1(t+C)-y_1(C))-x_1(C)y_1(C)\\&+\frac{1}{2}x_1(t+C)y_1(C)+\frac{1}{2}x_1(C)y_1(t+C)-x'_1(C)+\tilde{x}_0.
				\end{align*}
				From Equation \eqref{relx1x2} we have that $x_1(C)=\tilde{z}_0-z_0$ and $\tilde{x}_0=x'_2(0)=x'_1(C)$, then
				\begin{equation}\label{relz1z2}
					z_2(t)=z_1(t+C)-z_1(C)+\frac{1}{2}x_1(t+C)y_1(C)-\frac{1}{2}x_1(C)y_1(t+C).
				\end{equation}	 
				
				From Equations \eqref{relx1x2}, \eqref{rely1y2} and \eqref{relz1z2} we obtain that:
				\begin{equation}
					\sigma_2(t)=\sigma_1(C)^{-1}\sigma_1(t+C)
				\end{equation}  
				for all $t$.
				We conclude that $\sigma_2=((L_{\sigma_1(C)^{-1}},1),C)\cdot\sigma_1$ with $((L_{\sigma_1(C)^{-1}},1),C)\in H_F\ltimes \RR$.
				
			\end{proof}
			
			

			\section{Closed magnetic trajectories on compact nilmanifolds $\Lambda\backslash H_3$}\label{section6}
			
			The aim of this section is to study closed magnetic trajectories on compact  Heisenberg nilmanifolds, that is $M=\Lambda\backslash H_3$, where $\Lambda$ denotes a lattice in $H_3$, i.e., a discrete cocompact subgroup.  
			
			More generally, let $(N, \la\,,\ra)$ denote any Lie group endowed with a left-invariant metric and $\Lambda$ a lattice in $N$. The compact quotient $\Lambda \backslash N$ naturally inherits a Riemannian metric from $N$. Indeed, this is well defined because the metric is left-invariant, i.e., $\la \cdot, \cdot \ra_{np} =\la \cdot, \cdot \ra_p$ for all $n,p\in N$. 
			
			Moreover, any left-invariant magnetic field on $N$ induces a magnetic field on $\Lambda\backslash N$, and the same holds for the associated Lorentz forces. We will refer to these induced magnetic fields and Lorentz forces as left-invariant, identifying them with their corresponding lifts to $N$.    
			
			\begin{defn} \label{periodic}
				Let $N$ be a Lie group.  For any element  $\lambda \in N$ different from the identity, a curve  $\sigma(t)$ is called {\em $\lambda$-periodic} with period $\omega\neq 0$ if, for all $t \in \mathbb{R}$, it holds:
				\begin{equation}\label{per}
					\lambda \sigma(t)=\sigma(t+\omega).
				\end{equation}
				
			\end{defn}
			
			The following result shows properties of this notion. 
			
				\begin{prop}\label{lambperiodnyc}
				Let $\sigma$ denote a $\lambda$-periodic curve with period $\omega$ on a Lie group $N$ and fix $p\in N$. Then:
				\begin{enumerate}
					\item $\sigma$ is a $\lambda^m$-periodic curve with period $m\omega$ for any $m\in\ZZ-\{0\}$.
					\item $p\sigma$ is a $p\lambda p^{-1}$-periodic curve with period $\omega$.
				\end{enumerate}
			\end{prop}
			
			The first part is immediate from Equation \eqref{per} and the second holds from the next relation
			$$p\lambda p^{-1} p\sigma(t)=p\lambda \sigma(t)=p\sigma(t+\omega)\quad \mbox{ for every } t.$$

			It is clear that whenever the subgroup  $\Lambda<N$ is a lattice,  for any element  $\lambda \in \Lambda$, a $\lambda$-periodic magnetic trajectory projects to a smoothly closed magnetic trajectory under the mapping $N \rightarrow \Lambda \backslash N$ and will be contained in the free homotopy class corresponding to $\lambda$ (restricting the domain of the curve to a loop).  Conversely, every closed magnetic trajectory $\sigma$ on $\Lambda \backslash N$ lifts to a $\lambda$-periodic magnetic trajectory on the Lie group $N$, if $\sigma$ is non-contractible, or directly lifts to a closed magnetic trajectory on $N$, if $\sigma$ is contractible.

			Next we  study closed magnetic geodesics on compact quotients for which we need the notion above. Let $N$ be a simply connected $2$-step nilpotent Lie group equipped with a left-invariant metric and let $\nn$ denote its Lie algebra.  The metric on $\nn$  induces an orthogonal decomposition as direct sum of vector spaces:
			\begin{equation}\label{decomp}
				\nn = \vv \oplus \zz, \qquad \mbox{ with } \vv=\zz^{\perp}, 
			\end{equation}
			where $\zz$ denotes the center of $\nn$. Choose an element $\lambda\in N$, where  $\lambda=\exp(W_1+Z_1)$ with $W_1\in \vv, Z_1\in \zz$. 
			For $V(t)$ and $Z(t)$ curves on $\vv$ and $\zz$ respectively, a  trajectory $\sigma(t)=\exp(V(t)+Z(t))$ passing through the identity is $\lambda$-periodic with period $\omega$, for  $\omega\neq 0$, if and only if the following equations are verified 	for all $t\in\RR$:
			\begin{equation}\label{gamma-periodic2step}
				\left\{ \begin{array}{rcl}
					W_1+V(t) & = & V(t+\omega)\\
					Z_1+Z(t)+\frac{1}{2}[W_1,V(t)] & = & Z(t+\omega).
				\end{array} \right..
			\end{equation}	
		 In fact, this is possible since the exponential map is a diffeomorphism. 
			
			Let $p_{\vv}:\nn \to \vv$ and $p_{\zz}:\nn \to \zz$ denote the orthogonal projections to $\vv$ and $\zz$ respectively with respect to the decomposition \eqref{decomp}. Denote by $F_{\vv}$ and $F_{\zz}$ the following linear maps:
			$$F_{\vv} := p_{\vv} \circ F \qquad \qquad  F_{\zz}:= p_{\zz} \circ F.$$

			\begin{lem} Let $(N, \la\,,\,\ra)$ be a simply connected 2-step Lie nilpotent Lie group equipped with a left-invariant metric. Then a trajectory $\sigma=\exp(V(t)+Z(t))$ on $N$ is $\lambda=\exp(W_1+Z_1)$-periodic, for $W_1\in \vv, Z_1\in \zz$, if and only if Equations \eqref{gamma-periodic2step} hold. 
			\end{lem}

			In \cite{OS} it was proved that a magnetic trajectory on the 2-step nilpotent Lie group $N$ of the form $\sigma(t)=\exp(V(t)+Z(t))$, passing through the identity, satisfies the equations
			\begin{equation}\label{magnetic-2step}
				\left\{ \begin{array}{rcl}
					V''-j(Z'+\frac12 [V', V])V' & = & q F_{\vv}(V'+ Z'+ \frac12 
					[V',V])\\
					Z''+\frac12 [V'',V] & = & qF_{\zz}(V'+ Z'+ \frac12 
					[V',V]).
				\end{array} \right.
			\end{equation}
			
			Assume now that  $\dim(\zz/C(\nn))\leq 1$. In this situation the closedness condition of the magnetic field implies that  $F_{\zz}(\zz)=0$. Therefore if the curve $\sigma(t)=\exp(V(t)+Z(t))$ is a magnetic trajectory for a left-invariant Lorentz force $F$ on a 2-step nilpotent Lie group $(N, \la\,,\,\ra)$, then the second equation of \eqref{magnetic-2step} follows
			$$	Z''+\frac12 [V'',V] = qF_{\zz}(V'), $$
			which is equivalent to 
			\begin{equation}\label{equ-zmagnetic}
				Z'(t)+\frac12 [V'(t),V(t)]  = qF_{\zz}(V(t))+Z_0,
			\end{equation}
			for every $t$ and where $Z'(0)=Z_0$ is the initial velocity in the center of the magnetic trajectory. 
			By evaluating this  equation in $t+\omega$ and by using Equation \eqref{gamma-periodic2step} we get	
			\begin{equation*}\label{Zmagneticeq2}
				Z'(t)+\frac12 [W_1,V'(t)] + \frac12 [V'(t),W_1 +V(t)]  = qF_{\zz}(W_1+ V(t))+Z_0.
			\end{equation*}	
			By comparing the last two equations, one  obtains the following condition for $W_1$:
			\begin{equation}\label{Zmagnetic2}
				F_{\zz}(W_1)=0
			\end{equation}
			
			The converse of this reasoning is true as proved below. 
			
			\begin{lem} \label{lambdaper} Let $(N, \la \,,\, \ra)$ be a simply connected  2-step nilpotent Lie group such that $\dim (\zz/C(\nn))\leq 1$. Let $F$ be a left-invariant Lorentz force on $N$ with magnetic trajectory $\sigma(t)=\exp(V(t)+Z(t))$  through the identity element at $t=0$. Denote by $\lambda=\exp(W_1+Z_1)$  an element in $N$, with $W_1\in \vv, Z_1\in \zz$.
				
				\begin{enumerate}[(i)]
					\item  If $\sigma$ is $\lambda$-periodic of period $\omega$, then 
					$$W_1\in \ker F_{\zz}\quad \mbox{ and } \quad V' \mbox{ is periodic of period }\omega.$$
					\item Conversely, assume $V'(t)$ is periodic of period $\omega$ and $W_1:=V(\omega)$ belongs to the kernel of $F_{\zz}$. Then the magnetic curve $\sigma$ is $\lambda$-periodic for $\lambda=\exp(W_1+Z_1)$ where $Z_1:=Z(\omega)$.
				\end{enumerate}
			\end{lem}
			\begin{proof} To prove  (i) make use of Equation \eqref{Zmagnetic2}, while the  periodicity of the map $V'$ follows from the first equation in System \eqref{gamma-periodic2step}.
				
				(ii) Since $V'$ is periodic with period $\omega$, there exists $W_1$ such that $W_1+V(t)=V(t+\omega)$. Moreover note that $W_1=V(\omega)$. 
				
				
				Define a map $G:\RR \to \vv$ by $G(t)=Z(t+w)-Z(t)-\frac12 [W_1, V(t)]$. By differentiating, using Equation \eqref{equ-zmagnetic} and the formulas above, one gets
				$$ \begin{array}{rcl}
					G'(t) &=&	Z'(t+\omega)-Z'(t)-\frac12[W_1,V'(t)]  \\
					&=&  qF_{\zz}(V(t+w))+ Z_0-\frac12[V'(t+\omega), V(t+\omega)]-Z'(t)-\frac12 [W_1, V'(t)]\\
					&=& qF_{\zz}(V(t)) +  qF_{\zz}(W_1) +Z_0 - \frac12[V'(t), V(t)+W_1)]
					-Z'(t) - \frac12 [W_1, V'(t)]\\
					&=& qF_{\zz}(V(t)) + Z_0 - \frac12[V'(t), V(t)] -Z'(t)= 0.
				\end{array}
				$$
				
				This says that $G(t)$ is constant and equals
				$$Z_1:=Z(\omega)=Z(t+\omega)-Z(t)-\frac12[W_1,V(t)].$$
				This finally proves that $\sigma$ is $\lambda$-periodic for $\lambda=\exp(W_1+Z_1)$.
			\end{proof}
			
			\begin{rem} \label{remexact}
				For a Lorentz force corresponding to an exact form, under the hypothesis $\dim (\zz/C(\nn))\leq 1$, the condition $W_1\in \ker F_{\zz}$ is automatically satisfied (see \cite{OS2}).
			\end{rem}

			Now, the goal is to find  $\lambda$-periodic magnetic curves on the 3-dimensional Heisenberg Lie group for any left-invariant Lorentz force. The left-invariant exact case is treated on \cite{EGM}. For the other cases, we can restrict to the Lorentz force $F=F_{e_1, \rho}$ with $\rho\geq 0$, in view of  the next general lemma.
			
			Let $(N,\la\,,\,\ra)$ denote a Lie group with a left-invariant metric. As seen in Section \ref{symmetries},  there is an action of $\Iso(N)\times \RR$ on the set of curves, such that if $\gamma$ is a magnetic trajectory for the Lorentz force $F$ then $(\phi, r)\cdot \gamma$ is a magnetic trajectory for $(\phi,r)\cdot F=r\phi_*\circ F \circ \phi_*^{-1}$. 
			
			Assume  $\psi$ is an orthogonal automorphism of $N$ and suppose $\gamma$ is a $\lambda$-periodic magnetic trajectory. Thus, one has
			$$
			\psi(\lambda)\psi(\gamma(rt)) =	\psi(\lambda \gamma(rt)) =   \psi\gamma(rt+\omega),
			$$
			which equivalently says that  $(\psi,r)\cdot \gamma$ is $\psi(\lambda)$-periodic. This proves the next result.

			\begin{lem}\label{lemaequivlambdaperiod}
				Let $(N,\la\,,\,\ra)$ be a Lie group endowed with a left-invariant metric, and let $F$ be a Lorentz force on $N$. Suppose that $\psi$ is an orthogonal automorphism of $N$.  If $\gamma$ is a $\lambda$-periodic magnetic trajectory for $F$ with period $\omega$, then $(\psi,r)\cdot\gamma$ is a $\psi(\lambda)$-periodic magnetic trajectory with the same period $\omega$  for  the Lorentz force  $(\psi,r)\cdot F$.  
			\end{lem}

			In view of the previous result, we may fix  the left-invariant Lorentz force $F_{e_1,\rho}$ on the Heisenberg Lie group $H_3$.  Notice that the kernel of $F_{e_1, \rho }$ is spanned by $e_2+\rho e_3$. Thus, 	 Lemma \ref{lambdaper} and the fact that
			$F_{\zz}(W_1)=0$, implies $ W_1 \in span\{e_2\}$. Consequently, if $W_1=x_1e_1+y_1e_2$ and $\lambda=\exp(W_1+Z_1)$, then any $\lambda$-periodic trajectory necessarily satisfies $x_1=0$. 
			
			Let $\sigma(t)= \exp(x(t)e_1+y(t)e_2+z(t)e_3)$ be a  magnetic curve through the identity on $H_3$ for the Lorentz force $F_{e_1, \rho}$. The trajectory $\sigma$ is $\lambda$-periodic for $\lambda=\exp(y_1 e_2+z_1 e_3)$ if and only if for all $t\in\RR$, the next equations hold:
			\begin{equation*}\label{gammaperiodic}
				\begin{array}{rcl}
					x(t) & = & x(t+\omega)\\
					y(t)+y_1 & = & y(t+\omega)\\
					z(t)+z_1-\frac12 y_1 x(t) & = & z(t+\omega), 
				\end{array}
			\end{equation*}
			which is System \eqref{gamma-periodic2step} in this situation. Thus, $x$ must be periodic. Moreover, using the expression of $z(t)$ in Equation \eqref{eqmagneticHeisz}, we obtain
			$z_1=z(\omega)=-(z_0+\rho) y_1$. Hence, if $y_1=0$, then necessarily $z_1=0$ implying that $\lambda$ is the identity, which is a contradiction. Therefore, $y_1=y(\omega)\neq0$.
			
			Conversely, let $\sigma(t)= \exp(x(t)e_1+y(t)e_2+z(t)e_3)$ be a magnetic trajectory through the identity such that $x(t)=x(t+\omega)$ for every $t$, and assume that $y(\omega)\neq0$. Then $\sigma(t)$ is $\lambda$-periodic where $\lambda=\exp(y(\omega) e_2+z(\omega) e_3)$.
			
			Indeed, from the expression for $y(t)$ given in Equation \eqref{eqmagneticHeisy}, it follows that the second equation is satisfied, since $y'(t)$ has period $\omega$.  Moreover, using the expression for $z(t)$ from the same lemma, together with the fact that both $x(t)$ and $y'(t)$ are periodic functions with period $\omega$, we obtain
			$$\begin{array}{rcl}
				\left(z(t+w)-z(t)+\frac12 y_1 x(t)\right)' & = & 
				-\frac{1}{2}x'(t)y(t+\omega)   +	\frac{1}{2}x'(t)y(t) + \frac{1}{2}y_1 x'(t)\\
				&	= &\frac{1}{2}x'(t)\left(-y(t+\omega) + y(t) +y_1\right)=0
			\end{array}
			$$
			Hence, the third equation also holds. Comparing with Remark \ref{remperiod}, we derive the following result.
			
			\begin{prop}\label{proplambdaperiod1}
				Let  $(H_3, \la\,,\,\ra)$ be the Heisenberg Lie group endowed with its canonical metric, and let $F_{e_1, \rho}$ be the corresponding Lorentz force.  Consider the magnetic trajectory $\sigma$ through the identity satisfying $\sigma'(0)=x_0 e_1 + y_0 e_2 + z_0 e_3$. 
				\begin{itemize}
					\item If $x_0= 0$ and $(y_0+1)(z_0+\rho)= \rho$, then $\sigma$ is $\lambda_r$-periodic with $\lambda_r=\exp(y_0 r e_2+z_0 r e_3)$ for any $r\neq 0$. 
				\end{itemize}	
				In the case $x_0^2 + (\rho-(y_0+1)(z_0+\rho))^2\neq 0$,
				\begin{itemize}
					\item if $\Delta\neq 0$  or $\mu>0$ then $\sigma$ is periodic if $\sigma(\omega)=e$, the identity element or $\sigma$ is $\lambda$-periodic for $\lambda=\sigma(\omega)$ where $\omega$ is the corresponding period given in Remark \ref{remperiod}, 
					
					\item if $\Delta= 0$ and $\mu\leq0$ then $\sigma$ is not $\lambda$-periodic for any $\lambda$.  
				\end{itemize}
			\end{prop}
			
			Observe that if  $\sigma$ does not pass through the identity then, by Lemma \ref{lambperiodnyc},
			 the curve $p\sigma(t)$ with $p=\sigma(0)^{-1}$ is a $p\lambda p^{-1}$-periodic magnetic trajectory through the identity and 
			we proceed  as above.
			
			In the following, we fix an element $\lambda\in H_3$, different from the identity, and study the existence of $\lambda$-periodic magnetic geodesics on $H_3$.
			
			\begin{prop}\label{proplambdaperiod2}
					Let  $(H_3, \la\,,\,\ra)$ be the Heisenberg group endowed with its canonical metric, and let $F_{e_1, \rho}$ be the corresponding Lorentz force with $\rho\neq0$. Take $\lambda=\exp(x_1e_1+y_1e_2+z_1e_3)\in H_3$, with $(x_1,y_1,z_1)\neq (0,0,0)$, and $\mathcal{E}>0$,
				\begin{itemize}
					\item If $x_1\neq 0$ or $y_1=0$ then there does not exist any $\lambda$-periodic magnetic trajectory.
					\item If $x_1= 0$ and $y_1\neq0$, there are $\lambda$-periodic magnetic trajectories with energy $\mathcal E$.
				\end{itemize}	
			\end{prop}
			
			\begin{proof}
			Let $\lambda=\exp(y_1e_2+z_1e_3)$ with $y_1\neq 0$ and let $\mathcal E>0$. From the discussion above, it is enough to find a magnetic trajectory through the identity with energy $\mathcal E$ such that $x(t)=x(t+\omega)$ $\forall\, t$, $y(\omega)=y_1$ and $z(\omega)=z_1$ for some $\omega$. 
			
			Recall the function $\Psi(c,d,e)$ defined in Equation \eqref{ywfunc2} whose value coincides with $y(\omega)$ for the magnetic trajectory associated with the triple $(c,d,e)$ (see Equation \eqref{condinitcd}). Also, consider the $2$-dimensional surface $$\mathcal{S}_{\mathcal E}=\{(c,d,e): \En(c,d,e)=\mathcal E \}$$ where $\En(c,d,e)$ denotes the energy as a function of $(c,d,e)$, given in Equation  \eqref{Energycde}. 
			
			In Theorem \ref{thenergyperiodic}, we proved that there are exactly one $c_0>1$ and $0<d_{c_0}<1$ such that $(c_0,d_{c_0},e)\in\mathcal{S}_{\mathcal E}$ and $\Psi(c_0,d_{c_0},e)=0$ for every $e\in[-1,1]$. Solving the equation $\En(c,d,e)=\mathcal E$ for $d$ as a function of $c$, we obtain
			$$h_{\mathcal E}(c)=\sqrt{\frac{2c^4 {\mathcal E} - (c^4+\rho^2)(c^2-1)^2}{(c^4+\rho^2)4c^2}}$$
			which is defined on some neighborhood $J$ of $c_0$. Hence
			, $h_{\mathcal E}(c_0)=d_{c_0}$, and $(c,h_{\mathcal E}(c),e)\in\mathcal{S}_{\mathcal E}$ for every $c\in J$ and $e\in[-1,1]$. Also,
			{\small
				\begin{equation*}
					\begin{split}
						h'_{\mathcal E}(c_0)= -\frac{(c_0^4-\rho^2)2c_0^4{\mathcal E}+(c_0^4 -1)(c_0^4 + \rho^2)^2}{4 d_{c_0}(c_0^4+\rho^2)^2 c_0^4}&\leq -\frac{(c_0^4-\rho^2)(c_0^4+\rho^2)(c_0^2-1)^2+(c_0^4 -1)(c_0^4 + \rho^2)^2}{4 d_{c_0}(c_0^4+\rho^2)^2 c_0^4}\\
						&=-\frac{(c_0^2-1)(c_0^6 + \rho^2)(c_0^4 + \rho^2)}{4 d_{c_0}(c_0^4+\rho^2)^2 c_0^4}<0
					\end{split}
			\end{equation*}}where the first inequality holds since $h_{\mathcal E} (c_0)$ is a real number, i.e. $2c_0^4 {\mathcal E} - (c_0^4+\rho^2)(c_0^2-1)^2\geq 0$.
			Thus, the function $h_{\mathcal E}$ is decreasing in a neighborhood of $c_0$.  Consequently, there exist $c_1<c_0<c_2$ such that $h_{\mathcal E}(c_1)>d_{c_0}>h_{\mathcal E}(c_2)$. On the other hand,  we know that $c\to d_c$ is increasing (by Theorem \ref{thenergyperiodic}), so $d_{c_1}<d_{c_0}<d_{c_2}$. It follows that $h_{\mathcal E}(c_1)>d_{c_1}$ and $h_{\mathcal E}(c_2)<d_{c_2}$. Therefore, by Remark \ref{remposPsi}, we get that $\Psi(c_1,h_{\mathcal E}(c_1),e)<0$ and  $\Psi(c_2,h_{\mathcal E}(c_2),e)>0$. By continuity, we conclude that the image of $\Psi(c,d,e)$ restricted to the surface $ \mathcal{S}_{\mathcal E}$ contains an open neighborhood of $0$. More precisely, there exists $\varepsilon>0$ such that for any $y\in(-\varepsilon,\varepsilon)$ there is a magnetic trajectory $\exp(x(t)e_1+y(t)e_2+z(t)e_3)$ with energy $\mathcal E$, where $x(t)$ is periodic of period $\omega$ and $y(\omega)=y$.  
			
			Now, consider $n\in\NN$ such that $\frac{y_1}{n}\in(-\varepsilon,\varepsilon)$. By the previous argument, there is a $\lambda_1$-periodic magnetic trajectory through the identity with energy $\mathcal E$, where  $\lambda_1=\exp(\frac{y_1}{n}e_2+z_2e_3)$. Applying Proposition \ref{lambperiodnyc}, the same curve is a  $\lambda_1^n$-periodic magnetic trajectory through the identity with energy $\mathcal E$. Note that $\lambda_1^n=\exp(y_1e_2+nz_2e_3)$. Moreover, $\lambda_1^n$ and $\lambda=\exp(y_1e_2+z_1e_3)$ are conjugated elements. Therefore, another application of  Proposition \ref{lambperiodnyc} yields a  $\lambda$-periodic magnetic trajectory with energy $\mathcal E$.
		\end{proof}

	\section{The proof of Theorem 1}
	
	To complete the proof of Theorem 1, we now introduce the notion of the Ma\~n\'e critical value on $M=\Lambda \backslash H_3$ and compute it for every left-invariant Lorentz force on $H_3$ which is induced to $M$ in a usual way. 		
	
	Let $(M, g, \omega)$ be a closed, connected $n$-dimensional manifold endowed with a Riemannian metric $g$ and a magnetic field $\omega$. Consider the universal covering $\widetilde{M}$ of $M$, and let $\widetilde{\omega}$ denote the pullback of $\omega$ to $\widetilde{M}$. We say that $\omega$ is {\it weakly exact} whenever $\widetilde{\omega}$  is exact \cite{Me}.
	
	In this case,  define {\it the Mañe critical value} as \cite{Me}
	$$c(g,\omega)=\inf_{\theta : d\theta=\widetilde{\omega}} \sup_{p\in\widetilde{M}} \frac{1}{2}|\theta_p|^2$$

	This value is defined in a more general setting of Lagrangian systems, see \cite{Mn} or \cite{EGM}.  In the cases where $H^2(\widetilde{M})=0$ every magnetic field is weakly exact, in particular, this holds if $\widetilde{M}=H_3$.  Observe that, by definition,
$c(g,\omega)$ is finite if and only if $\omega$ admits a bounded primitive $\theta$.

We now provide an elementary proof to compute the Mañé critical value for invariant magnetic fields on a Heisenberg nilmanifold.

\begin{prop}\label{critical}
		Let $M=\Lambda \backslash H_3$ be a compact Heisenberg nilmanifold. An invariant magnetic field $\omega_F$ on $M$ has finite Mañe's critical value if and only if $\omega_F=\rho e^1\wedge e^2$ for some $\rho\in\RR$. In this case, $\omega_F$ is invariant exact and $c(g,\omega_F)=\frac{\rho^2}{2}$.  
\end{prop}	

\begin{proof}
	Suppose that an invariant magnetic field $\omega_F = \rho e^1 \wedge e^2 + \beta e^1\wedge e^3 + \alpha e^2\wedge e^3$ admits a bounded primitive on $H_3$; that is, there exists a $1$-form $\theta$ such that
	$$d\theta = \omega_F \qquad \text{and} \qquad  \sqrt{2c(g,\omega_F)}\leq K_{\theta}:=\sup_{x\in H_3} |\theta_x|<\infty.$$

	Let $R>0$. Consider the curve $\gamma_R:[0,2\pi]\to H_3$, given by $\gamma_R(t)=(R\cos(t),0,R\sin(t))$  and the immersion of the ball $B_R=\{(x,z)\in\RR^2:x^2+z^2<R^2\}$ into $H_3$, $i: B_R\to H_3$ given by $i(x,z)=(x,0,z)$. 
	
	By Stoke's theorem, using that $i^*e^1=dx$, $i^*e^2=0$ and $i^*e^3=dz$, we obtain
	$$\int_{\gamma_R} \theta = \int_{i(B_R)} d\theta =\int_{i(B_R)} \omega_F =\int_{B_R} i^*\omega_F = \int_{B_R} \beta dx\wedge dz = \beta \pi R^2$$
	
	Since $|\gamma'(t)|=|-R\sin(t)e_1+Rcos(t)e_3|=R$ for all $t$, for the invariant metric given on $H_3$. Then the length of $\gamma_R$ is  $2\pi R$ and we have that
	$$\left|\beta \pi R^2\right| = \left|\int_{\gamma_R} \theta\right|\leq K_{\theta} \, \text{length}(\gamma_R) =K_{\theta} \, 2\pi R.$$
	
	Since $R$ is arbitrary, we conclude that $\beta=0$. An analogous argument, applied to the $yz$-plane in $H_3$, shows that $\alpha =0$.
	
   If we now consider the $xy$-plane, an analogous argument yields 
	 $$\left|\rho \pi R^2\right| \leq K_{\theta} \,  \text{length}(\gamma_R),$$
	but in this case, $|\gamma'(t)|=|-R\sin(t)e_1+R
	\cos(t)e_2-\frac{R^2}{2}e_3|=\frac{R}{2}\sqrt{R^2+4}$. So that
	 $$\frac{|\rho|}{K_{\theta}}\leq \frac{\sqrt{R^2+4}}{R}$$
	for every $R>0$. Letting $R\to\infty$, we conclude that $K_{\theta}\geq \rho$. Since $c(g,\omega_F)=\inf_\theta(\frac{K_{\theta}^2}{2})$, we have that $c(g,\omega_F)\geq \frac{\rho^2}{2}$.

	Finally observe that $\omega_F = \rho e^1 \wedge e^2 = d(\rho e^3)$ so  $c(g,\omega_F)\leq \sup_{x\in H_3} \frac{1}{2}|\rho e^3_x|^2=\frac{\rho^2}{2}$. Thus $\omega_F = \rho e^1 \wedge e^2$ and $c(g,\omega_F)=\frac{\rho^2}{2}$.
\end{proof}	

\begin{rem}
	The first part of the previous proposition can be derived as a consequence of Corollary 5.4 in \cite{PG} and Nomizu's theorem and the second part has been proved in \cite{EGM}. 
\end{rem}


\subsection{The proof and examples}
We are now in a position to prove Theorem 1. 

We begin by introducing a notion that describes the possible homotopy classes of low-energy magnetic trajectories. 
		\smallskip
		\begin{defn}\label{defFadm}
			Let $F$ be a non-trivial left-invariant Lorentz force on $H_3$ and let $\lambda \in H_3$, $\lambda$ different from the identity, written as $\lambda=\exp(W+Z)$ with $W\in\vv$ and $Z\in\zz$. We say that $\lambda$ is {\it $F$-admissible} if:
		$$
		\begin{array}{rcll}
				{W }& = & 0, \qquad &\text{ for } F=F_{0,\rho},	\\
					F_\zz(W) & = & 0 \text{ and } W\neq 0, \qquad & \text{ for } F\neq F_{0,\rho}.
	\end{array}
		$$
		\end{defn}
  
  \smallskip
  
	\begin{proof}[Proof of Theorem 1]
	
	
	
We provide the proof in terms of left-invariant forces. 

\smallskip 

{\em Case $F=F_{0,\rho}$. } 
	This case is treated in \cite{EGM}. Since $c=\frac{\rho^2}{2}$ as shown in Proposition \ref{critical}, part $(i)$ of Theorem 1 is a consequence of Theorem \ref{periodicexact} (Theorem 1 in \cite{EGM}), while part $(ii)$ follows from Theorem 3 and Lemma 7 in \cite{EGM}. 

\smallskip

{\em Case $F \neq F_{0,\rho}$. } By Proposition \ref{orbits} the force $F$ is equivalent to $F_{e_1,\rho}$ under the action of an element $(\psi, r)\in \Iso(M) \times \RR^*$.

 By Lemmas \ref{lem2} and \ref{lemaequivlambdaperiod}, closed and $\lambda$-periodic magnetic trajectories on $H_3$ for $F_{e_1,\rho}$ correspond to closed and $\psi(\lambda)$-periodic magnetic trajectories on $H_3$ for the force $F$. Moreover, one readily checks that $\lambda$ is  $F_{e_1,\rho}$-admissible if and only if $\psi(\lambda)$ is $F$-admissible. 
 
  By Proposition \ref{critical}, we have $c=\infty$. Hence, part $(i)$ follows from Theorem \ref{thenergyperiodic}, while part $(ii)$ follows from Proposition \ref{proplambdaperiod2}.
\end{proof}

\smallskip

We finish the work by exhibiting lattices for which the quotient space 
$\Lambda\backslash H_3$ either admits infinitely many non-contractible magnetic trajectories or none at all. 

Recall that any lattice on $H_3$ is isomorphic to one in  the family $$\Gamma_k =\ZZ\times\ZZ\times \frac{1}{2k}\ZZ \mbox{ \quad for } k\in\NN.$$
A proof can be read in \cite{GW}.  

\begin{example}\label{exampleInfiniteperiodic}
	Fix the Lorentz force $F_{e_1, \rho }$. For a lattice $\Gamma_k$ as above, take $\lambda_p=(0,p,\frac{1}{2k})\in\Gamma_k$ for every prime $p\in\ZZ$, all $F_{e_1, \rho }$-admissible elements, and let $\sigma_p$ be the corresponding $\lambda_p$-periodic magnetic trajectory for the invariant Lorentz force . Then $\sigma_p$ is different of $\sigma_q$ for $p\neq q$ by Lemma \ref{minperiod} below. Thus the corresponding projections to the compact space $\Gamma_k\backslash H_3$ induce different closed curves.
\end{example}

\begin{lem}\label{minperiod}
	Let $\Lambda$ be a lattice in a Lie group $N$ and $\sigma$ a curve in $N$ such that its projection to the compact space $\Lambda\backslash N$ is a periodic trajectory. Then there exist a unique $\lambda_0\in \Lambda$ and $0<\omega_0\in\RR$ such that $\sigma$ is  $\lambda_0$-periodic of period $\omega_0$, and verifies the following condition:
	
	If $\sigma$ is  $\lambda$-periodic of period $\omega$ with $\lambda\in\Lambda$ then  $\lambda=\lambda^n_0$ and $\omega=n\omega_0$ for some $n\in\ZZ$. 
\end{lem}

\begin{proof}
	It is clear that the sets $$\mathcal{L}=\{\lambda\in \Lambda: \sigma \mbox{ is $\lambda$-periodic}\}\cup\{e \} \mbox{ and }$$
	$$ \mathcal{W}=\{\omega\in\RR: \omega \mbox{ is a period of $\sigma$ for some $\lambda\in\Lambda$}\}\cup\{0 \}$$
	are nonempty subgroups of $H_3$ and $\RR$, respectively. Let see that $0$ is an isolated point of $\mathcal{W}$.
	
	Suppose, by contradiction, that there exists a sequence $\{\omega_n\}_{n\in\NN}\subset\mathcal{W}$ such that $\omega_n\rightarrow 0$. Let $\{\lambda_n\}_{n\in\NN}\subset\mathcal{L}$ be the corresponding sequence. Then
	$$\lambda_n \sigma(t)=\sigma(t+\omega_n) \rightarrow \sigma(t) \mbox{ which implies } \lambda_n \to e.$$
	This is impossible since $e\in\Lambda$ and $\Lambda$ is discrete.
	
	Therefore $\mathcal{W}$ is a discrete subgroup of $\RR$. Hence, there exists $\omega_0>0$ such that $\mathcal{W}=\ZZ \omega_0$. 
	
	Finally, observe that the map assigning $\lambda\in\mathcal{L} \to \omega\in\mathcal{W}$ is a group isomorphism. It follows that $\mathcal{L}$ is the cyclic discrete subgroup generated by $\lambda_0$ corresponding to $\omega_0$.
	
\end{proof}

\begin{example}\label{exampleNoContract}	
	Consider the lattice $$\Lambda = \left\{(v_1,v_2,\frac{z}{2k}): z\in\ZZ \mbox{ and } 
	\left(\begin{matrix}
		v_1\\
		v_2
	\end{matrix}\right)=  \left(\begin{matrix}
		\frac{\sqrt{3}}{2} &  \frac{1}{2}\\
		-\frac{1}{2}  & \frac{\sqrt{3}}{2}
	\end{matrix}\right)  \left(\begin{matrix}
		m\\
		n
	\end{matrix}\right)
	\mbox{ for } m,\,n\in\ZZ\right\}.$$
	In this case, in order  that $x_1=0$ holds, one must have $\frac{\sqrt{3}}{2} m + \frac{1}{2}n =0 \Leftrightarrow m=n=0$. Therefore, $\Lambda$ does not admit any $F_{e_1, \rho }$-admissible element. Consequently, there are no non-contractible closed magnetic trajectory on $\Lambda\backslash H_3$ for the Lorentz force $F_{e_1, \rho}$.
	
\end{example}

On the other hand, in view of Remark \ref{remexact}, when the Lorentz force $F$ corresponds to an invariant exact form, it is not possible to construct a lattice $\Lambda$ with no $F$-admissible element.

	\begin{cor}
		Let $(H_3, g)$ be the Heisenberg Lie group endowed with a left-invariant metric, and let $F$ be a left-invariant Lorentz force. If $F\neq F_{0,\rho}$, then there exist lattices $\Gamma\subset H_3$ such that every closed magnetic trajectory on $\Gamma\backslash H_3$ is contractible.   
	\end{cor}
	
	
	
	\section{Appendix}
	This section collects the main properties of elliptic integrals used throughout the paper. 
	Recall that the functions
	$$F(x,k)=\int_0^x \frac{d\theta}{\sqrt{1-k^2 \sin^2 \theta}}  \qquad  E(x,k)=\int_0^x \sqrt{1-k^2 \sin^2 \theta}\, d\theta $$ 
	$$\Pi(x,\alpha^2,k)=\int_0^x \frac{d\theta}{(1-\alpha^2 \sin^2\theta)\sqrt{1-k^2 \sin^2 \theta}}$$
	with parameters $0<k<1$ and $\alpha^2\in\RR$, are the incomplete elliptic integrals of the first, second and third kind, respectively. And the corresponding complete elliptic integrals are defined as
	$$K(k)=F\left(\frac{\pi}{2},k\right), \quad  E(k) = E\left(\frac{\pi}{2},k\right) \ \mbox{ and } \ \Pi(\alpha^2,k)=\Pi\left(\frac{\pi}{2},\alpha^2,k\right). $$

	Then the Jacobi elliptic functions are given by:
	$$\mathrm{am}(x,k)= F^{-1}(x,k), \ \  \mathrm{sn}(x,k) = \sin(\mathrm{am}(x,k)), \ \  \mathrm{cn}(x,k) = \cos(\mathrm{am}(x,k))$$
	and
	$$\mathrm{dn}(x,k)=\frac{d}{dx} \mathrm{am}(x,k).$$

We summarize below several primitives and integral identities involving these functions. They can be derived and verified from the definitions and identities collected in \cite{BF}.

	$$\int_{0}^{t} \mathrm{cn}^2(s,k) ds =\frac{E(\mathrm{am}(t, k), k)-(1-k^2)t}{k^2} \  \  \  \  \  \   \  \int_{0}^{t} \mathrm{dn}^2(s,k) ds =E(\mathrm{am}(t, k), k)$$

	\begin{equation*}
		\int_{0}^{t} \frac{1}{1+A \mathrm{sn}^2(s,k)}ds =\Pi(\mathrm{am}(t,k),-A,k)
	\end{equation*}

	\begin{equation*}
		\begin{split}
			\int_{0}^{t} \frac{1}{(1+A\, \mathrm{sn}^2(s,k))^2}ds=&\frac{A^2\, \mathrm{sn}(t,k) \mathrm{cn}(t,k) \mathrm{dn}(t,k)}{2(A+1)(A+k^2)(1+A\,\mathrm{sn}^2(t,k))} +\frac{A\,E(\mathrm{am}(t,k),k)}{2(A+1)(A+k^2)}\\
			&+\frac{(2Ak^2+A^2+3k^2+2A)\Pi(\mathrm{am}(t,k),-A,k)}{2(A+1)(A+k^2)}-\frac{x}{2(A+1)}
		\end{split}	
	\end{equation*}

	
	\begin{equation}\label{Equ46}
		\int \limits_{0}^{2K(k)}  \frac{1}{1+A\,\mathrm{sn}^2(s,k)}ds=2\Pi\left(-A, k\right) \qquad \mbox{ and }
	\end{equation}
	\begin{equation}\label{Equ47}
		\int \limits_{0}^{2K(k)}  \frac{1}{(A\,\mathrm{sn}(s,k)^2 + 1)^2}ds=\frac{A E(k)}{(A+1) (A+k^2)}+\frac{2Ak^2+A^2+3k^2+2A}{(A+1)(A+k^2)}\Pi(-A,k)-\frac{K(k)}{A+1}.
	\end{equation}
	
	For $0< A^2 < B^2$, one can check the following primitives
	\begin{align*}
		\int \limits  \frac{1}{A\mathrm{cn}(s,k) + B}ds = \frac{ B}{A_1^2}\,\Pi\left(\mathrm{am}(s,k),-\frac{A^2}{A_1^2},k\right)
		-\frac{A}{A_1 B_1}\arctan\left(\frac{B_1 \mathrm{sn}(s,k)}{A_1\mathrm{dn}(s,k)}\right)	
	\end{align*}
	and
	\begin{align*}
		\int \limits  \frac{1}{(A\mathrm{cn}(s,k) + B)^2}ds =& \frac{B^2((1-2k^2)A^2+2k^2B^2)}{A_1^4 B_1^2}\, \Pi\left(\mathrm{am}(s,k),-\frac{A^2}{A_1^2},k\right)\\
		&-\frac{AB((1-2k^2)A^2+2k^2B^2)}{ A_1^3 B_1^3}\arctan\left(\frac{B_1 \mathrm{sn}(s,k)}{A_1\mathrm{dn}(s,k)}\right)\\
		&+\frac{A^2}{A_1^2 B_1^2} E(\mathrm{am}(s, k),k) -\frac{A^3 \mathrm{sn}(s,k)\mathrm{dn}(s,k)}{A_1^2 B_1^2 (A\mathrm{cn}(s,k)+B)}-\frac{s}{A_1^2}.
	\end{align*}
	where $A_1^2=B^2-A^2$ and $B_1^2 = (1-k^2)A^2+k^2B^2$.
	
	Evaluating the definite integrals, we get 
	\begin{equation}\label{int1cn}
		\int \limits_{0}^{4K(k)}  \frac{1}{A\mathrm{cn}(s,k) + B}ds = \frac{4 B}{B^2 - A^2} \Pi\left(-\frac{A^2}{B^2-A^2}, k\right)
	\end{equation}
	and 
	\begin{equation}\label{int1cn2}
		\begin{split}
			\int \limits_{0}^{4K(k)}  \frac{1}{(A\mathrm{cn}(s,k) + B)^2}ds  = & \frac{4 B^2((1-2k^2)A^2+2 k^2 B^2 )}{(B^2-A^2)^2((1-k^2)A^2 + k^2 B^2)} \Pi\left(\frac{A^2}{A^2-B^2}, k\right) \\
			& +\frac{4A^2 E(k)}{(B^2-A^2)((1-k^2)A^2 + k^2 B^2)}-\frac{4K(k)}{B^2-A^2},
		\end{split}
	\end{equation}
	where we use that $\mathrm{sn}(x,k)$,  $\mathrm{cn}(x,k)$ and  $\mathrm{dn}(x,k)$ have period $4K(k)$ or $2K(k)$. In particular, by evaluating the last integrals at $k=0$, we obtain
		\begin{equation}\label{int1cos}
			\int \limits_{0}^{2\pi}  \frac{1}{A\cos(s) + B}ds = \sgn(B)\frac{2 \pi}{\sqrt{B^2 - A^2}}, \quad \int \limits_{0}^{2\pi}  \frac{1}{(A \cos(s) + B)^2}ds = \frac{2 |B| \pi}{(B^2-A^2)^{3/2} },
		\end{equation}

			\end{document}